\documentclass[12pt]{amsart}

\usepackage{amssymb, amsthm, tikz}
\usepackage{MnSymbol, mathtools}
\usepackage{fullpage}
\usepackage{hyperref}
\usepackage{cleveref}
\usepackage{afterpage}

\definecolor{green}{RGB}{0,180,0}

\hypersetup{
	pdftitle={Combinatorics of Iwahori Whittaker Functions},
	pdfauthor={Slava Naprienko},
	linktocpage=false,          
	colorlinks=true,            
	linkcolor=blue!75!black,    
	citecolor=green!50!black,   
	filecolor=magenta,          
	urlcolor=blue!75!black      
}

\newtheorem*{theorem*}{Theorem}
\newtheorem{theorem}{\textbf{Theorem}}[section]
\newtheorem{proposition}[theorem]{\textbf{Proposition}}

\newtheorem{lemma}[theorem]{\textbf{Lemma}}

\theoremstyle{definition}
\newtheorem{definition}[theorem]{Definition}
\newtheorem{remark}[theorem]{Remark}
\newtheorem{example}[theorem]{Example}
\newtheorem{procedure}[theorem]{Procedure}

\numberwithin{equation}{section}

\newcommand{\OO}{\mathcal{O}}  
\newcommand{\p}{\mathfrak{p}}  
\newcommand{\Z}{\mathbb{Z}}  
\newcommand{\N}{\mathbb{N}}  
\newcommand{\C}{\mathbb{C}}  
\DeclareMathOperator{\Ind}{Ind} 
  
\DeclareMathOperator{\GL}{GL}  
\DeclareMathOperator{\wt}{wt} 
\DeclareMathOperator{\val}{val} 
\DeclareMathOperator{\Lu}{Lu} 
\DeclareMathOperator{\LLu}{\mathbf{Lu}} 
\DeclareMathOperator{\SLLu}{\mathbf{SLu}} 
\DeclareMathOperator{\SSLLu}{\mathbf{SSLu}} 
\DeclareMathOperator{\GT}{GT} 
\DeclareMathOperator{\GGT}{\mathbf{GT}} 
\DeclareMathOperator{\SGGT}{\mathbf{SGT}} 
\DeclareMathOperator{\G}{G} 

\tikzstyle{spin} = [circle, draw, fill=white, minimum size=14pt, inner sep=0pt, text=black]
\tikzstyle{spins}=[every node/.style={spin}]
\tikzstyle{path}=[line width=1.5pt, spins]

\newcommand{\grid}[2]{

    \foreach \i in {0,...,#1}{
        \pgfmathtruncatemacro\col{int(#1-\i)}
        \draw (2*\i+1,0) -- (2*\i+1,2*#2) node[spin, label=above:$\col$] {$+$};
    }

    \foreach \i in {1,...,#2}{
        \pgfmathtruncatemacro\row{#2+1-int(#2-\i+1)}
        \draw (0,2*\i-1) node[spin, label=left:$\row$] {$+$} -- (2*#1+2, 2*\i-1);
    }

    \foreach \i in {0,...,#1}{
        \foreach \j in {1,...,#2}{
            \node[spin] at (2*\i+1, 2*\j-2) {$+$};
            \node[spin] at (2*\i+2, 2*\j-1) {$+$};
        }
    }
}

\begin{document}
	\title{Combinatorics of Iwahori Whittaker Functions}
	\author{Slava Naprienko}
	\email{slava@naprienko.com}
	\urladdr{\href{https://naprienko.com/}{naprienko.com}}
	\let\thefootnote\relax\footnote{The most updated version of the paper is available on \href{naprienko.com/papers/miw/}{https://naprienko.com/papers/miw/}}
	
	
	\maketitle
	
	\begin{abstract}
	    We give a combinatorial evaluation of Iwahori Whittaker functions for unramified genuine principal series representations on metaplectic covers of the general linear group over a non-archimedean local field. To describe the combinatorics, we introduce new combinatorial data that we call colored data: colored Lusztig data, colored Gelfand-Tsetlin patterns, and colored lattice models. We show that all three are equivalent. To achieve the result, we give an explicit Iwahori decomposition for the maximal unipotent subgroup of a split reductive group which gives the parametrization of the generalized Mirkovi\'c-Vilonen cycles in the affine flag varieties and is of interest in itself. 
	    
	    Our result is based and naturally extends Peter McNamara's evaluation of the metaplectic spherical Whittaker function in terms of Lusztig data.
	\end{abstract}
	
	\tableofcontents
	
	\section{Introduction}
	
	Consider a split reductive group over a non-archimedean local field. The principal series representation is a representation parabolically induced from a character of the maximal torus. We consider Whittaker functions on the group which are certain matrix coefficients on the principal series representation corresponding to Whittaker functionals and spherical/Iwahori vectors. See \Cref{section:miwappendix} for details. 
	
	The non-achimedean Whittaker functions were introduced by Jacquet in \cite{J67} and were given in the form of a $p$-adic Jacquet integral over the maximal unipotent subgroup. There is a rich theory of combinatorial evaluation of the non-achimedean Whittaker functions explored in \cite{McN11, BBF11, BBBG19, BBBG20} and related special functions in \cite{Tok88, Ok90, HK07} and others. 
	
	The spherical Whittaker function $W_0$ is the first and simplest example. It was computed explicitly by T. Shintani in \cite{Sh76} for Lie groups of type A. Later Casselman and Shalika in \cite{CS80} gave a formula for all Lie types as a product of the deformed Weyl denominator times the character of the dual Langlands group. 
	
	The spherical Whittaker functions $W_0^n$ of the metaplectic $n$-covers for Lie groups of type A is the next example. In \cite{McN11}, McNamara gave the expression of $W_0^n$ in terms of Lusztig data and showed that this data is equivalent to Mirkovi\'c-Vilonen cycles of the affine grassmanian. In \cite{BBF11}, Brubaker, Bump, and Friedberg also expressed $W_0^n$ (as the Whittaker coefficient of Borel Eisenstein series) in terms of crystal graphs (and in terms of strict Gelfand-Tsetlin patterns). 
	
	The Iwahori Whittaker functions $W_w$ in type A are refinements of the spherical Whittaker functions that comes from the Iwahori decomposition instead of the Iwasawa decomposition. In \cite{BBBG19}, a recursive evaluation of the functions $W_w$ was given using the intertwining operators. By connecting the intertwiners with Yang-Baxter equations, functions $W_w$ were expressed in terms of colored lattice models. Moreover,  In \cite{BBBG20}, the calculation was extended for metaplectic $n$-covers of Lie groups of type A in terms of supersymmetric lattice models using similar methods. At the present, it is the most general result which specializes to all functions mentioned above. 
	
	Thanks to the Casselman-Shalika formula (the CS-formula) for the spherical Whittaker function $W_0$, one can study the combinatorics by analyzing the explicit expression given in terms of the special functions. In this approach, one does not need any representation theory to obtain the combinatorial results. These special functions occur independently in the literature with no relation to the Whittaker functions.
	
	Tokuyama in \cite{Tok88} found an exact expression for the generating function of strict Gelfand-Tsetlin patterns which incidentally matched the CS-formula. In \cite{Ok90}, the CS-formula was expressed in terms of partially strict shifted plane partitions. In \cite{HK07}, it was computed combinatorically by applying jeu de taquin to the primed tableaux. In \cite{BBS11}, Brubaker, Bump, and Friedberg evaluated the CS-formula as the partition function of a free-fermionic solvable lattice model which gave a new combinatorial interpretation using Yang-Baxter equation, a tool from statistical mechanics. A posteriori, all these methods give a combinatorial evaluation of the spherical Whittaker function $W_0$. 
	
	Some progress was made for other types other Lie types as well. In \cite{HK02}, a variation of the CS-formula was given combinatorically by Hamel and King. Moreover, they introduced lattice models with U-turns. In \cite{Iv12}, Ivanov used lattice models with U-turns and introduced Yang-Baxter equation that allowed to get the CS-formula for the symplectic group as the partition function of his model. Later, in \cite{Gr17} Gray extended lattice models with U-turns and conjectured that the partition function is the spherical Whittaker function for the metaplectic covers of the symplectic group. In \cite{FZ16}, a partial evaluation of the CS-formula for type B was found by Friedberg and Zhang. Also, in \cite{DF21} DeFranco gave the CS-formula for type $G_2$ by brute force. A non-finite variant of such computations was given for the metaplectic cover of $G_2$ by Leslie in \cite{Les19}.
	
	\begin{center}
    \begin{table}
    \begin{tabular}{ cc c } 
    & Spherical vector & Iwahori vector \\
    \hline
    Type A & \cite{Tok88}, \cite{Ok90}, \cite{HK07}, \cite{BBS11} & \cite{BBBG19}, this paper\\ 
    Metaplectic type A & \cite{BBF11}, \cite{McN11}  & \cite{BBBG20}, this paper\\ 
    Type C & \cite{HK02} (partial), \cite{Iv12} & -- \\
    Metaplectic type C & \cite{Gr17} (partial) & -- \\
    Type B & \cite{FZ16} (partial) & -- \\
    Type G2 & \cite{DF21} & --\\
    Other types & -- & -- \\
    \\
    \end{tabular}
    
    \caption{Combinatorial evaluations of the Whittaker functions and related special functions given by the Casselman-Shalika formula}
    \end{table}
    \end{center}
	
	The main goal of this paper is to extend McNamara's approach from the spherical Whittaker function to the Iwahori Whittaker functions. An advantage of McNamara's approach is the description of the spherical Whittaker function as a sum over combinatorial data. Each term in the sum corresponds to a Mirkovi\'c-Vilonen cycle. The combinatorial sum is obtained directly, without using intertwining operators, recursive relations, explicit expressions, or multiple Dirichlet series. This geometric approach has the following steps. 
	
	\begin{enumerate}
	    \item Write explicit Iwasawa decomposition for the maximal unipotent subgroup $U$ that corresponds to Mirkovi\'c-Vilonen cycles in the affine grassmanian. It gives the decomposition $U = \bigsqcup_m C_m^{\textbf{i}}$ into cells;
	    \item The spherical Whittaker function $W_0^n$ is defined by a $p$-adic integral over $U$. The explicit decomposition for $U$ gives a combinatorial sum $\sum_m \int_{C_m^{\textbf{i}}}$;
	    \item Evaluate each individual integral $\int_{C_m^{\textbf{i}}}$ explicitly in appropriate coordinates to obtain combinatorial weights for each term;
	    \item Describe the decomposition and weights combinatorically, for example, in terms of Gelfand-Tsetlin patterns.
	\end{enumerate}
	
	In Section 4 of \cite{McN11}, McNamara found an explicit Iwasawa decomposition for the maximal unipotent subgroup $U$ into disjoint cells  $C_m^{\textbf{i}}$. In \Cref{section:decomposition}, we extend his decomposition and produce explicit Iwahori decomposition for $U$ into cells $S_{m,\sigma}^{\textbf{i}}$ that are the refinement of McNamara's decomposition. 
	
	In Section 8 of \cite{McN11}, McNamara computes the values of the spherical Whittaker functions using the explicit decomposition. In \Cref{section:calculation} we extend the calculation to the case of the Iwahori Whittaker functions using cells $S_{m,\sigma}^{\textbf{i}}$. 
	
	There Iwahori case has several difficulties compared to the spherical case from McNamara's work. First, the space of Iwahori Whittaker functions is not one dimensional which results in computing each Iwahori basis element separately. Second, the Iwahori Whittaker functions are defined by the affine Weyl group, not just by the diagonal which results in two parameters for the group element to consider. Third, we introduce new combinatorial data that we call colored data. We give equivalent definitions of colored Lusztig data, colored Gelfand-Tsetlin patterns, and colored lattice models. 
	
	The lattice model we introduce is dual to the one from \cite{BBBG19, BBBG20}. This gives a geometric meaning to the admissible states of the lattice model as they are in bijection with the cells $S_{m,\sigma}^{\textbf{i}}$ and the weights are the $p$-adic integrals over the corresponding cells. We remark that we don't use the intertwining operators or Yang-Baxter equation in our approach. Instead, we give the combinatorial evaluation directly from the definition of the Iwahori Whittaker functions by the Jacquet integral.
	
	Let us be more precise now.

    Let $G$ be a split reductive group over a non-achimedean local field and let $\widetilde{G}$ be an $n$-fold metaplectic cover of $G$. Let $\widetilde{T}$ be a fixed split maximal torus of $\widetilde{G}$, $\widetilde{B}$ the Borel subgroup, $U^{\pm}$ be the maximal unipotent subgroup and its opposite, $J$ the opposite Iwahori subgroup, and $W$ the Weyl group. See \Cref{section:metaplecticgroups} for details. In \Cref{section:decomposition}, we give the explicit Iwahori decomposition for $U$ that depends on a long word decomposition $\textbf{i}$:
    
    \begin{theorem*}[\Cref{thm:iwahoridecomposition}]
        Let $u \in U$ and $w' \in W$. Then there is an explicit Iwahori decomposition $uw' = p_1h_1\pi_1j_1 \in U^+\widetilde{T}WJ$ with $p_1 \in U^+$, $h_1 \in \widetilde{T}$, $\pi_1 \in W$, $j_1 \in J$ given by \Cref{lem:iwahoridecompositionstep}.
    \end{theorem*}
    
    We use this explicit decomposition to write $U$ as a union of cells which are related to the generalized Mirkovi\'c-Vilonen cycles in the affine flag varieties. In this paper, it will be important that we can describe each cell in appropriate coordinates to compute integrals explicitly. 
	
	\begin{theorem*}[\Cref{thm:celldecomposition}]
	    The following decomposition holds. 
	    \[
	        U \cap \widetilde{B}w_0wJ(w')^{-1} = \bigsqcup_{\substack{m \in \N^N\\ \sigma \in \Sigma(m,w,w')}}S_{m,\sigma}^{\textbf{i}},
	    \]
	    where colorings $\Sigma(m,w,w')$ are defined by \Cref{def:cells}.
	\end{theorem*}
	
	Similar to Section 7 of \cite{McN11}, we can show that cells $S_{m,\sigma}^{\textbf{i}}$ are in bijection with the generalized Mirkovi\'c-Vilonen cycles in the affine flag variety which gives a parametrization of the generalized Mirkovi\'c-Vilonen cycles. Alternative parametrization in terms of refined alcove paths was explored in \cite{PRS09} by Parkinson, Ram, and Schwer. 
	
	\begin{remark}
	    We remark that our parametrization of the generalized Mirkovi\'c-Vilonen cycles resembles the Lusztig graph which has the natural crystal structure. It is interesting to find the ``colored crystal structure'' on the generalized Mirkovi\'c-Vilonen cycles in the affine flag variety. 
	\end{remark}
	
	Let now specialize to $G = \GL_{r+1}$. In \Cref{section:iwahorifunctions} we define Iwahori Whittaker functions $W_w\colon \widetilde{G} \to \C$ parametrized by Weyl group elements $w \in W$. In short, these are the matrix coefficients on the unramified genuine principal series representation of $\widetilde{G}$ corresponding to the averaged Whittaker functional and an Iwahori fixed vector. In \Cref{section:calculation}, we evaluate all values of $W_w$ in terms of colored Lusztig data, combinatorial objects we introduce in \Cref{section:coloreddata}.  

    \begin{theorem*}[\Cref{thm:miw}]
        Let $\lambda \in \Lambda$ with $\lambda_{r+1} = 0$ and let $w, w' \in W$. Then the integrals $\phi_w(\lambda, w'; z)$ defined by \cref{eq:miw} that determines values of Iwahori Whittaker functions is given by 
        \begin{equation*}
            \phi_w(\lambda, w'; z) = \sum_{\textbf{m} \in \LLu(\lambda+\rho,w,w')}\prod_{\alpha \in \Phi^+}\G(\textbf{m},\alpha; q)z^{m_\alpha \alpha},
        \end{equation*}
        where contributions $\G(\textbf{m},\alpha; q) \in \Z[q^{-1}]$ are given explicitly by \eqref{eq:lusztigcontribution}.
    \end{theorem*}
    
    Compare it with Theorem 8.6 from \cite{McN11}:
    \begin{theorem*}[Theorem 8.6, \cite{McN11}]
        Let $\lambda \in \Lambda$ with $\lambda_{r+1} = 0$. Then the integrals $\phi_0(\lambda; z)$ defined in the paper that determines values of the spherical Whittaker function is given by 
        \[
            \phi_0(\lambda; z) = \sum_{\textbf{m} \in \Lu(\lambda+\rho)}\prod_{\alpha \in \Phi^+}w(m,\alpha;q)z^{m_\alpha \alpha},
        \]
        where contributions $w(m,\alpha; q)$ are given explicitly by (8.2) in \cite{McN11}.
    \end{theorem*}
    
    As usual, one can use the combinatorial formulas to get the branching rules, specializations, and asymptotics for the Iwahori Whittaker functions. We do not do it here. 
    
    \begin{remark}
        Note that $\phi_0(\lambda;z) = \sum_{w \in W}\phi_w(\lambda,w';z)$. Therefore, \Cref{thm:miw} is a refinement of McNamara's calculation. We remark that in the summation, non-trivial cancellations between Iwahori Whittaker functions happen which are ``invisible'' in the spherical case. See \Cref{section:examples} for examples.
    \end{remark}
    
    In \Cref{section:coloreddata}, we define colored Lusztig data, colored Gelfand-Tsetlin patterns, and supersymmetric lattice models. We show that all of them are equivalent. Thus, \Cref{thm:miw} can be written over any of these equivalent data. In examples in \Cref{section:examples} we primarily use colored GT-patterns.
    
    \begin{theorem}[\Cref{thm:bijection}]
        There are weight-preserving bijections between
        \begin{itemize}
            \item colored Lusztig data $\LLu(\lambda+\rho,w,w')$,
            \item colored Gelfand-Tsetlin patterns $\GGT(\lambda+\rho,w,w')$,
            \item colored states $\mathfrak{S}(\lambda+\rho,w,w')$.
        \end{itemize}
    \end{theorem}
    
    This result connects our approach to colored lattice models and supersymmetric lattice models dual to ones introduced in \cite{BBBG19} and \cite{BBBG20}. We note that the alternative choice of the long word decomposition yields the exact lattice models. By combining two models, one can obtain the standard results like the Cauchy identities and duality for the supersymmetric ice. 
    
    \begin{remark}
        Note that the Boltzmann weights for the lattice models are given by $p$-adic integrals over the corresponding cells. A posteriori, the lattice model is integrable and satisfies the Yang-Baxter equation. In other words, this approach allows one to come up with solvable lattice models and Boltzmann weights without ``guessing'' the weights.
    \end{remark}
    
    We finally remark that our approach can be adapted in other settings. Instead of the general linear group, one can choose any other split reductive group and its covers. Instead of the delta word decomposition, one can pick the gamma word decomposition to obtain combinatorics matching colored lattice models from \cite{BBBG19, BBBG20}. The main difficulty lies in the explicit expression of the character $\psi_\lambda$ in coordinates which is the reason why we need to pick a ``good'' word decomposition.  
    
    We use notation $\N = \Z_{\geq 0}$.
    
    \bigskip
    
    \textbf{Acknowledgements.} We thank Dan Bump for being a great advisor and his support throughout the project. We thank Alexei Borodin and Ben Brubaker for their helpful comments on the presentation of the paper. We thank Claire Frechette for the discussion about the general metaplectic covers. Thank you! 
    
    \section{Explicit Iwahori Decomposition}\label{section:decomposition}
    In Section 4 of \cite{McN11}, McNamara gives an explicit Iwasawa decomposition for the opposite maximal unipotent subgroup of a metaplectic cover of a split reductive group. We extend his results to get an explicit \textit{Iwahori} decomposition. We loosely follow the exposition of Section 4 of \cite{McN11}. 
    
    Let $G$ be a split reductive group over a non-archumedian local field $F$. Let $\widetilde{G}$ be an $n$-fold metaplectic cover of $G$. We use notation from \Cref{section:metaplecticgroups}, \ref{section:iwahorifunctions}: ring of integers $\OO_F$, maximal ideal $\p$, torus $\widetilde{T}$, Borel subgroup $\widetilde{B}$, maximal unipotent subgroup $U^{+}$ and its opposite $U = U^-$, and the opposite Iwahori subgroup $J$. And from \Cref{section:generators}: positive roots $\Phi^+$, coroots $\Phi^\vee$, the Weyl group $W$, simple reflections $s_i$, and generators $e_\alpha$ and $h_\alpha$.
    
    Let $\textbf{i} = (i_1, \dots, i_N)$ be a $N$-tuple of indices such that $w_0 = s_{i_1}s_{i_2}\dots s_{i_N}$ is a reduced decomposition of the long word $w_0 \in W$. The choice of $\textbf{i}$ defines a total ordering $ <_{\textbf{i}} $ on the set of positive roots given by
    \[
    \Phi^+ = \{\gamma_1, \gamma_2, \dots, \gamma_{N}\}, \quad \gamma_1 <_{\textbf{i}} \gamma_2 <_{\textbf{i}} \dots <_{\textbf{i}} \gamma_{N},
    \]
    where $\gamma_j = s_{i_N}s_{i_{N-1}}\dots s_{i_{j+1}}\alpha_{i_j}$. The proof can be found in any Lie theory text, for example \cite{B68}[Ch VI, §6].
    
    For each $k = 0, 1, 2, \dots, N$, let $G_k$ denote the set of elements $g \in \widetilde{G}$ which can be expressed in the form
    \[
        g = \left(\prod_{j=1}^{k-1}e_{-\gamma_j}(t_j)\right)\left(\prod_{j=k+1}^{N}e_{\gamma_j}(t_j)\right), \quad \text{all $t_j \in F$}.
    \]
    Note that $G_0 = U^+$. It slightly differs from $G_k$ from Section 4 of \cite{McN11} as we don't include the torus element in our definition.
    
    \begin{lemma}[Lemma 4.3 from \cite{McN11}]\label{lem:commute}
        For all $z \in F$ and $g \in G_k$, there exists a unique $g' \in G_k$ such that 
        \[
            e_{-\gamma_k}(z)g = g'e_{-\gamma_k}(z).
        \]
    \end{lemma}
    \begin{proof}
        In the proof of Lemma 4.3 from \cite{McN11} the argument $z$ changes only when crossing a torus element which we don't have in our definition of $G_k$.
    \end{proof}
    
    Before we give an explicit Iwahori decomposition, we need a technical result which is an extension of Algorithm 4.4 of \cite{McN11} to the Iwahori subgroup.
    
    \begin{proposition}\label{lem:iwahoridecompositionstep}
        Let $u \in U$ and $w' \in W$. Write $u \in U$ as $u = e_{-\gamma_1}(x_1)\dots e_{-\gamma_N}(x_N)$ for unique $x_1,\dots,x_N \in F$. Then there exist coordinates $y_1,y_2,\dots,y_{N}$, shifts $t_1,t_2,\dots,t_N$, and elements $\{p_k,h_k,\pi_k,j_k\}_{k=1}^{N+1}$ with $p_k \in e_{-\gamma_{k-1}}(t_{k-1})G_{k-1}$, $h_k \in \widetilde{T}$, $\pi_k \in W$, and $j_k \in J$, such that for each $k=1,2,\dots,N+1$ we have
        \[
            uw' = \left(\prod_{j=1}^{k-1}e_{-\gamma_j}(x_j)\right)p_k h_k \pi_k j_k,
        \]
        and explicit expressions
        \[
            h_k = \prod_{j=N}^{k}\begin{cases}
                h_{\gamma_j}(y_j^{-1}), \quad &\text{if $y_k \not\in \OO$ or if ($y_k \in \OO^\times$ and $\pi_{k+1}^{-1}(\gamma_k) \in \Phi^-$)} \\
                1, \quad &\text{if $y_k \in \p$ or if ($y_k \in \OO^\times$ and $\pi_{k+1}^{-1}(\gamma_k) \in \Phi^+$)}
            \end{cases},
        \]
        
        \[
            \pi_k = \prod_{j=k}^{N}\begin{cases}
                s_{\gamma_j}, &\text{if $y_k \not\in \OO$ or if ($y_k \in \OO^\times$ and $\pi_{k+1}^{-1}(\gamma_k) \in \Phi^-$)} \\
                1, \quad &\text{if $y_k \in \p$ or if ($y_k \in \OO^\times$ and $\pi_{k+1}^{-1}(\gamma_k) \in \Phi^+$)}
            \end{cases},
        \]
        \[
            j_k = \prod_{j=k}^{N}\begin{cases}
                e_{\sigma_{j+1}^{-1}(\gamma_j)}(y_j^{-1}), \quad &\text{if $y_k \not\in \OO$ or if ($y_k \in \OO^\times$ and $\pi_{k+1}^{-1}(\gamma_k) \in \Phi^-$)} \\
                e_{-\sigma_{j+1}^{-1}(\gamma_j)}(y_j), \quad &\text{if $y_k \in \p$ or if ($y_k \in \OO^\times$ and $\pi_{k+1}^{-1}(\gamma_k) \in \Phi^+$)}
            \end{cases}.
        \]
        \[
            y_k = (x_k+t_k)\prod_{j=k+1}^{N}\begin{cases}
                y_j^{-\langle \gamma_k, \gamma_j^\vee \rangle}, \quad &\text{if $y_k \not\in \OO$ or if ($y_k \in \OO^\times$ and $\pi_{k+1}^{-1}(\gamma_k) \in \Phi^-$)}\\
                1, \quad &\text{if $y_k \in \p$ or if ($y_k \in \OO^\times$ and $\pi_{k+1}^{-1}(\gamma_k) \in \Phi^+$)} \\
            \end{cases}
        \]
    \end{proposition}
    \begin{proof}
        By decreasing induction on $k$. For $k= N+1$, we have $p_{N+1} = 1$, $h_{N+1} = 1$, $\pi_{N+1} = w'$, $t_{N} = 0$, and $j_{N+1} = 1$. Next, suppose 
        \[
            uw' = \left(\prod_{j=1}^{k}e_{-\gamma_j}(x_j)\right)p_{k+1} h_{k+1} \pi_{k+1} j_{k+1}. 
        \]
        We take the last term $e_{-\gamma_k}(x_k)$ in the product and commute it with elements on the right. We write $p_{k+1} = e_{-\gamma_k}(t_k)p_{k}'$ for some $p_{k}' \in G_k$. By \Cref{lem:commute},
        \[
            e_{-\gamma_k}(x_k)p_{k+1} = e_{-\gamma_k}(x_k+t_k)p_{k}' = p_{k}''e_{-\gamma_k}(x_k+t_k),
        \]
        for some new $p_{k}'' \in G_{k}$. Next, commute with $h_{k+1}$ to get $e_{-\gamma_k}(x_k+t_k)h_{k+1} = h_{k+1}e_{-\gamma_k}(y_k)$. To commute with $\pi_{k+1}$, we consider the following cases.
        
        \begin{description}
            \item[If $y_k \not\in \OO$ or if ($y_k \in \OO^\times$ and $\pi_{k+1}^{-1}(\gamma_k) \in \Phi^-$)] We apply the identity 
            \[
                e_{-\gamma}(z) = h_{\gamma}(z^{-1})e_{\gamma}(z)s_{\gamma}e_{\gamma}(z^{-1}).
            \]
            to $e_{-\gamma_k}(y_k)$ and get 
            \[
                e_{-\gamma_k}(y_k)\pi_{k+1} = h_{\gamma}(y_k^{-1})e_{\gamma}(y_k)s_{\gamma_k}e_{\gamma_k}(y_k^{-1})\pi_{k+1} = h_{\gamma_k}(y_k^{-1})e_{\gamma_k}(y_k)s_{\gamma_k}\pi_{k+1}e_{\pi_{k+1}^{-1}(\gamma_k)}(y_k^{-1}).
            \]
            
            Let $h_{k+1}h_{\gamma_k}(y_k^{-1})e_{\gamma_k}(y_k)(h_{k+1}h_{\gamma_k}(y_k^{-1}))^{-1} = e_{\gamma_k}(b)$, then we define 
            \begin{align*}
                p_k &= p_{k}''e_{\gamma_k}(b) \in e_{-\gamma_{k-1}}(t_{k-1})G_{k-1} \\
                h_{k} &= h_{k+1}h_{\gamma_k}(y_k^{-1}) \in \widetilde{T} \\
                \pi_k &= s_{\gamma_{k}}\pi_{k+1} \in W \\
                j_k &= e_{\sigma_{k+1}^{-1}(\gamma_k)}(y_k^{-1})j_{k+1} \in J.
            \end{align*}
            
            \item[If $y_k \in \p$ or if ($y_k \in \OO^\times$ and $\pi_{k+1}^{-1}(\gamma_k) \in \Phi^+$)] Then $e_{-\gamma_k}(y_k)\pi_{k+1} = \pi_{k+1}e_{-\pi_{k+1}^{-1}(\gamma_k)}(y_k)$. 
            
            Then we define 
            \begin{align*}
                p_k &= p_{k}'' \in G_{k-1} \\
                h_{k} &= h_{k+1} \in \widetilde{T}\\
                \pi_k &= \pi_{k+1} \in W \\
                j_k &= e_{-\pi_{k+1}^{-1}(\gamma_k)}(y_k)j_{k+1} \in J.
            \end{align*}
        \end{description}
        In any case, the induction step is successful, and we are done. 
    \end{proof}
    
    \begin{theorem}[Explicit Iwahori Decomposition]\label{thm:iwahoridecomposition}
        Let $u \in U$ and $w' \in W$. Then there is an explicit Iwahori decomposition $uw' = p_1h_1\pi_1j_1 \in U^+\widetilde{T}WJ$ with $p_1 \in U^+$, $h_1 \in \widetilde{T}$, $\pi_1 \in W$, $j_1 \in J$ given by \Cref{lem:iwahoridecompositionstep}.
    \end{theorem}
    \begin{proof}
        Set $k = 1$ in \Cref{lem:iwahoridecompositionstep} and notice that $p_1 \in G_0 = U^+$.
    \end{proof}
    
    We can also reformulate the result in terms of simple reflections. We favor this reformulation because it is easier to give a combinatorial description in terms of simple reflections.
    
    \begin{lemma}\label{lem:simplerelfections}
        Define 
        \[
            \sigma_{k} = \left[\prod_{j=k}^{N}\begin{cases}
                1, &\text{if $y_k \not\in \OO$ or if ($y_k \in \OO^\times$ and $\pi_{k+1}^{-1}(\gamma_k) \in \Phi^-$)} \\
                s_{i_j}, \quad &\text{if $y_k \in \p$ or if ($y_k \in \OO^\times$ and $\pi_{k+1}^{-1}(\gamma_k) \in \Phi^+$)}
            \end{cases}\right]w'.
        \]
        Then $\pi_k = \left(\prod_{j=N}^{k}s_{i_j}\right)\sigma_k$, in particular, $\pi_1 = w_0\sigma_1$. Then we can rewrite expressions for $h_k, \pi_k, j_k$ from the theorem in terms of simple reflections.
    \end{lemma}
    \begin{proof}
        Recall that $\gamma_{k} = s_{i_N}s_{i_{N-1}}\dots s_{i_{k+1}}\alpha_{i_k}$, and so 
		\[
		s_{\gamma_{k}} = (s_{i_N}s_{i_{N-1}}\dots s_{i_{k+1}})s_{i_k}(s_{i_{k+1}}^{-1}\dots s_{i_{N-1}}^{-1} s_{i_N}^{-1}).
		\]
		Note that $s_{i_N}s_{i_{N-1}}\dots s_{i_{k}} = s_{\gamma_{k}}s_{\gamma_{k+1}}\dots s_{\gamma_N}$. 
		
		We prove $\pi_k = \left(\prod_{j=N}^{k}s_{i_j}\right)\sigma_{k}$ by decreasing induction on $k$. For $k= N+1$ we have $\pi_{N+1} = \sigma_{N+1} = w'$. Next, if $\pi_{k+1} = \left(\prod_{j=N}^{k+1}s_{i_j}\right)\sigma_{k+1}$, we consider two cases. 
		
		If $y_k \in \p$ or if ($y_k \in \OO^\times$ and $\pi_{k+1}^{-1}(\gamma_k) \in \Phi^+$), we get 
		\[
		    \pi_{k} = \pi_{k+1} = \left(\prod_{j=N}^{k+1}s_{i_j}\right)\sigma_{k+1} = \left(\prod_{j=N}^{k}s_{i_j}\right)\sigma_{k}.
		\]
		
		If $y_k \not\in \OO$ or if ($y_k \in \OO^\times$ and $\pi_{k+1}^{-1}(\gamma_k) \in \Phi^-$), we get 
		\[
		    \pi_{k} = s_{\gamma_{k}}\pi_{k+1} = s_{\gamma_k}\left(\prod_{j=N}^{k+1}s_{i_j}\right)\sigma_{k+1} = \left(\prod_{j=N}^{k}s_{i_j}\right)\sigma_{k}.
		\]
		We also note that
		\[
		    \pi_{k+1}^{-1}(\gamma_k) = \left(\sigma_{k+1}^{-1}\prod_{j={k+1}}^{N}s_{i_j}\right)\left(s_{i_N}s_{i_{N-1}}\dots s_{i_{k+1}}\alpha_{i_k}\right) = \sigma_{k+1}^{-1}(\alpha_{i_k}).
		\]
		Now all the data can be written in terms of the simple reflections. 
    \end{proof}
    
    Now we extract the combinatorial essence from the results above. We want to describe all possible sequences of $\sigma_1, \sigma_2, \dots, \sigma_{N+1}$ that can occur in the explicit Iwahori decomposition. We drop variables $y_k$ and keep only their valuations $m_k = \max(\val(y_k^{-1}), 0)$, in other words, if $y_k \not\in \OO$, we take the valuation of its inverse, and if $y_k \in \OO$, we set $m_k$ to be zero. 
    
    \begin{definition}\label{coloringsDefinition}
    	Let $m \in \N^N$ and let $w, w' \in W$. An $(m,w,w')-$\textit{coloring} is a sequence $(\sigma_1, \dots, \sigma_{N+1}) \in W^{N+1}$ with $\sigma_{1} = w$, $\sigma_{N+1} = w'$, and 
    	\begin{equation}\label{coloringRules}
    		\sigma_k = \begin{cases}
    			\sigma_{k+1}, &\text{if $m_k > 0$}\\
    			s_{i_k} \sigma_{k+1}, &\text{if $m_k = 0$ and $\sigma_{k+1}^{-1}(\alpha_{i_k}) \in \Phi^+$}\\
    			\text{$\sigma_{k+1}$ or $s_{i_k}\sigma_{k+1}$}, &\text{if $m_k = 0$ and $\sigma_{k+1}^{-1}(\alpha_{i_k}) \in \Phi^-$}
    		\end{cases}.
    	\end{equation}
    	Denote by $\Sigma(m,w,w')$ the set of $(m,w,w')$-colorings. 
    \end{definition}
    
    We call $w'$ the \textit{input}, and $w$ the \textit{output} of the coloring. 
    
    Informally, start with $w' \in W$ and then read $(m_1,m_2,\dots,m_N)$ from right to left. At each entry, update the current element. If $m_k > 0$, skip it. If $m_k = 0$, multiply by $s_{i_k}$ if applying the inverse of the current element to $\alpha_{i_k}$ lies in $\Phi^+$. Otherwise, either multiply or skip it. Then $\Sigma(m,w,w')$ is the set of such sequences that end with $w \in W$. Note that $\Sigma(m,w,w')$ may be empty.
    
    \begin{remark}
        Note that the coloring depends on the vanishing of entries of $m$, that is, whether an entry is zero or not, but not on the non-zero values of $m_k$'s.
    \end{remark}
    
    \begin{example}
        Let $G = \GL_3$ and let $\textbf{i} = (2,1,2)$. Then the positive roots are
        \[
            \Phi^+ = \{(1,2), (1,3), (2,3) \} \text{ with } (1,2) <_\textbf{i} (1,3) <_\textbf{i} (2,3).
        \] 
        Let $w' = 1$ be the input. Here are examples of different $(m,w,w')$-colorings.
        
        \begin{description}
            \item[$m = (0,0,0)$] $(s_2s_1s_2,s_1s_2,s_2,1)$ has output $s_2s_1s_2 = w_0$
            \item[$m = (0,0,1)$] $(s_2s_1,s_1,1,1)$ has output $s_2s_1$
            \item[$m = (0,1,0)$] $(1,s_2,s_2,1)$ has output $1$ and $(s_2,s_2,s_2,1)$ has output $s_2$
            \item[$m = (0,1,1)$] $(s_2,1,1,1)$ has output $s_2$
            \item[$m = (1,0,0)$] $(s_1s_2,s_1s_2,s_2,1)$ has output $s_1s_2$
            \item[$m = (1,0,1)$] $(s_2s_1,s_1,1,1)$ has output $s_2s_1$
            \item[$m = (1,1,0)$] $(s_2,s_2,s_2,1)$ has output $s_2$
            \item[$m = (1,1,1)$] $(1,1,1,1)$ has output $1$.
        \end{description}
    \end{example}
    
    \begin{definition}\label{def:cells}
        For $m \in \N^N$ and a coloring $\tau \in \Sigma(m,w,w')$ we define cells $S_{m,\tau} = S_{m,\tau}^{\textbf{i}}$ by
        \[
            S_{m,\tau} = \{u \in U \mid \max(\val(y_k^{-1}),0) = m_k \text{ and } \sigma_k = \tau_k \},
        \]
        using coordinates introduced in \Cref{thm:iwahoridecomposition} and \Cref{lem:simplerelfections}.
    \end{definition}
    
    \begin{theorem}\label{thm:celldecomposition}
        The following decomposition holds.
        \[
            U \cap Bw_0wJ(w')^{-1} = \bigsqcup_{\substack{m \in \N^N \\ \sigma \in \Sigma(m,w,w')}}S_{m,\sigma}.
        \]
    \end{theorem}
    \begin{proof}
        We first show that $S_{m,\sigma} \subset Bw_0wJ(w')^{-1}$. Let $u \in S_{m,\sigma}$. By definition in local coordinates of \Cref{thm:iwahoridecomposition}, we have $\pi_1 = w_0\sigma_1 = w_0w$, hence, $u \in Bw_0wJ(w')^{-1}$.
    
        Now we show that $u \in U \cap Bw_0wJ(w')^{-1}$ lies in $S_{m,\sigma}$ for some $m \in \N^N$ and coloring $\sigma \in \Sigma(m,w,w')$. Let $u \in Bw_0wJ(w')^{-1}$. We use the coordinates from \Cref{thm:iwahoridecomposition} to get explicit Iwahori decomposition to define $m$ and $\sigma$. 
        
        We set $m_k = \max(\val(y_k^{-1}), 0)$. Next, we set $\sigma_k$ to be the sigmas introduced in \Cref{lem:simplerelfections}. Note that $\pi_1 = w_0\sigma_1$ and from $u \in Bw\pi_1J(w')^{-1}$ and $u \in Bw_0\sigma_1J(w')^{-1}$ by Iwahori decomposition, we conclude that $\sigma_1 = w$. Next, $\sigma$ is indeed a $(m,w,w')$-coloring: if $m_k > 0$, then $y_k \in \p$, and so $\pi_k = \gamma_k\pi_{k+1}$ and $\sigma_k = \sigma_{k+1}$. If $m_k = 0$ and $\sigma_{k+1}^{-1}(\alpha_{i_k}) \in \Phi^+$, then $\pi_k = \pi_{k+1}$ and $\sigma_k = s_{i_k}\sigma_{k+1}$. If $m_k = 0$ and $\sigma_{k+1}^{-1}(\alpha_{i_k}) \in \Phi^-$, then we either have $\sigma_k = s_{i_k}\sigma_{k+1}$ in case $y_k \in \p$, or $\sigma_k = \sigma_{k+1}$ in case $y_k \in \OO^\times$. Therefore, $\sigma$ is a $(m,w,w')$-coloring, and we are done.
    \end{proof}
    
    \begin{remark}
        Similar to Section 7 in \cite{McN11}, we can show that the cells $S_{m,\sigma}^{\textbf{i}}$ parametrize the generalized Mirkovi\'c-Vilonen cycles in the affine flag varieties. Alternative parametrization in terms of refined alcove paths was given in \cite{PRS09} by Parkinson, Ram, and Schwer. 
    \end{remark}
    
    Write $y_k = \varpi^{-m_k}u_k$, $u_k \in \OO_F\times$ for $y_k \not\in \OO_F$ and $y_k = u_k$, $u_k \in \OO_F$ for $y_k \in \OO_F$. From the proof of \Cref{thm:celldecomposition}, we find a cell $S_{m,\sigma}$ is a set of elements in $U$ with $u_k \in D_k$, where the domains $D_{k} = D_{k}(m,\sigma)$ are defined by 
    \[
        D_k \coloneqq \begin{cases}
        	\OO^\times, &\text{if $m_k > 0$}\\
        	\OO, &\text{if $m_k = 0$, $\sigma_{k+1}^{-1}(\alpha_{i_k}) \in \Phi^+$}\\
        	\OO^\times, &\text{if $m_k = 0$, $\sigma_{k+1}^{-1}(\alpha_{i_k}) \in \Phi^-$, and $\sigma_k = \sigma_{k+1}$}\\
        	\p, &\text{if $m_k = 0$, $\sigma_{k+1}^{-1}(\alpha_{i_k}) \in \Phi^-$, and $\sigma_k = s_{i_k}\sigma_{k+1}$}
        \end{cases}.
    \]
    
    Now we give the relation to the cells introduced by McNamara in \cite{McN11}.
    
    \begin{theorem}[\cite{McN11}, Section 4]\label{thm:mcn} The following decomposition of $U$ holds. 
        \[
            U = \bigsqcup_{m \in \N^N}C_m^{\textbf{i}},
        \]
        where cells $C_m = C_m^{\textbf{i}}$ are defined in Section 4 of \cite{McN11}.
    \end{theorem}
    
    We have the following relation of our cells with McNamara's cells. 
    \begin{proposition}\label{lem:mcnamaracellrelation}
        McNamara cells decomposes into cells $S_{m,\sigma}$ as follows: 
        \[
            C_m \cap Bw_0wJ(w')^{-1} = \bigsqcup_{\sigma \in \Sigma(m,w,w')}S_{m,\sigma}.
        \]
    \end{proposition}
    
    \begin{remark}
        This connection of cells $S_{m,\sigma}$ explains the splitting phenomena observed in the colored lattice models. Since the integral over a cell $C_m$ equals to the sum of integrals over correspodning $S_{m,\sigma}$, we get a transition from colored models to uncolored models.
    \end{remark}
    
    \section{Colored data}\label{section:coloreddata}
    In the case of the general linear group and a good choice of a long word decomposition, it is possible to write down colorings from \Cref{section:decomposition} combinatorically which will be useful when computing Iwahori Whittaker functions in \Cref{section:calculation}. In that section we will need weights and contributions of the colored data. For coherency, we define them in this section next to the corresponding data. The statistics $r_{ij}$ and $s_{ij}$ and the Gauss sums $g(r_{ij},s_{ij})$ are defined in \Cref{section:calculation}.
    
    Let $G = \GL_{r+1}$. Realize the root system as $\Phi = \{(i,j) \in \Z^2 \mid 1 \leq i < j \leq r+1 \}$ and the positive roots as $\Phi^+ = \{(i,j) \in \Z^2 \mid 1 \leq i < j \leq r+1 \}$. Denote the order of $\Phi^+$ by $N = r(r+1)/2$. We identify weights $\Lambda = X_*(T) \cong \Z^{r+1}$ and use the standard elementary basis $e_i$, so that simple roots $\alpha_i = e_i - e_{i+1}$. The fundamental weights $\varpi_i$ are defined by 
	\[
        \varpi_i = (\underbrace{1, 1, \dots, 1}_{\text{$i$ times}}, \underbrace{0, \dots, 0}_{\text{$r+1-i$ times}}), \quad i = 1,\dots,r+1.
    \]
    
    We write a weight $\lambda = (\lambda_1, \lambda_2, \dots, \lambda_{r+1})$ uniquely as $\lambda = \sum_i \Lambda_i \varpi_i$ with $\Lambda_i = \lambda_i - \lambda_{i+1}$.
    
    Let $\textbf{i} = \Delta$ be the delta word, where \[
    	\Delta = (r,r-1,r,r-2,r-1,r,\dots,1,2,\dots,r).
    \]
    The choice of the delta word induces a total ordering $<_\Delta$ on the positive roots given by $(i,j) <_\Delta (i',j')$ if $j < j'$ or if $j = j'$ and $i < i'$. I.e.,
    \[
    	(1,2) <_\Delta (1,3) <_\Delta (2,3) <_\Delta (1,4) <_\Delta (2,4) <_\Delta \dots <_\Delta (r,r+1).
    \]
    
    We choose the delta word because it allows us to use the induction argument. To emphasize it, we write data labeled by positive roots in the form of a table, reading the data from the end and writing in the table from right to left, from top to bottom. For example,
    \[
    	\Delta = \begin{pmatrix}
    		1 & 2 & \dots & r-1 & r\\
    		& 2 & \dots & r-1 & r\\
    		& &\ddots & \vdots & \vdots\\
    		& & & r-1 & r \\
    		& & & & r
    	\end{pmatrix}, \quad 
    	\Phi^+ = \begin{pmatrix}
    		(1,r+1) & (2,r+1) & \dots & (r-1,r+1) & (r,r+1)\\
    		& (1,r) & \dots & (r-2,r) & (r-1,r)\\
    		& &\ddots & \vdots & \vdots\\
    		& & & (1,3) & (2,3) \\
    		& & & & (1,2)
    	\end{pmatrix}.
    \]
    
    We identify $W \cong S_{r+1}$. For the delta word, we interpret colorings in terms of actions on the set $\{1,2,\dots,r+1\}$; in this context we refer to numbers $1,2,\dots,r+1$ as \textit{colors}. 
    
    The results of the section are summarized in the following
    
    \begin{theorem}[Equivalence of Colored Data]\label{thm:bijection}
        There are weight-preserving bijections between
        \begin{itemize}
            \item colored Lusztig data $\LLu(\lambda+\rho,w,w')$,
            \item colored Gelfand-Tsetlin patterns $\GGT(\lambda+\rho,w,w')$,
            \item colored states $\mathfrak{S}(\lambda+\rho,w,w')$.
        \end{itemize}
    \end{theorem}
    
    \subsection{Colored Lusztig data}
    Let us rewrite \Cref{coloringsDefinition} for $(m,w,w')$-colorings in terms of permutations $w,w' \in S_{r+1}$. Note that $w(\alpha_i) \in \Phi^+$ if and only if $w(i) < w(i+1)$. 
    
    Let $m \in\N^N$ and let $w,w' \in S_{r+1}$ be two permutations. An \textit{$(m,w,w')$-coloring} is a sequence $\sigma = (\sigma_{12}, \sigma_{13}, \dots, \sigma_{r,r+1}, \sigma') \in S_{r+1}^{N+1}$ with $\sigma_{12} = w$ and $\sigma' = w'$. For the uniform notation, we denote $\sigma_{i,i} = \sigma_{1,i+1}$ for $i\in 1, 2,\dots,r$ and $\sigma' = \sigma_{1,r+2} = \sigma_{r+1,r+1}$. Then \cref{coloringRules} becomes 
    \begin{equation}\label{eq:instrutionsrules}
    	\sigma_{i,j} = \begin{cases}
    		\sigma_{i+1,j}, &\text{if $m_{i,j} > 0$}\\
    		s_{i} \sigma_{i+1,j}, &\text{if $m_{i,j} = 0$ and $\sigma_{i+1,j}^{-1}(i) < \sigma_{i+1,j}^{-1}(i+1)$}\\
    		\text{$\sigma_{i+1,j}$ or $s_{i}\sigma_{i+1,j}$}, &\text{if $m_{i,j} = 0$ and $\sigma_{i+1,j}^{-1}(i) > \sigma_{i+1,j}^{-1}(i+1)$}
    	\end{cases}.
    \end{equation}
    
    From now on, we consider only the colorings for the delta word for the general linear group. We call $\sigma' = w'$ the \textit{input} and $\sigma = w$ the \textit{output} of a coloring $\sigma$. With the new notation, we give a characterization of the colorings in the following
    
    \begin{lemma}\label{lem:coloring}
    	An $(m,w,w')$-coloring $\sigma$ is uniquely determined by $\{\sigma_{1,r+1-j+2}(r+1-k+2)\}$ with $j = 1, 2, \dots, r+1$ and $k = j,\dots,r+1$. We organize this data in table with $r+1$ rows:
    	\begin{equation}\label{eq:colortable}
    		\begin{pmatrix}
    			\sigma_{1,r+2}(1) & \sigma_{1,r+2}(2) & \dots & \sigma_{1,r+2}(r) & \sigma_{1,r+2}(r+1)\\
    			& \sigma_{1,r+1}(2) & \dots & \sigma_{1,r+1}(r) & \sigma_{1,r+1}(r+1)\\
    			& & \ddots & \vdots & \vdots\\
    			& & & \sigma_{1,3}(r) & \sigma_{1,3}(r+1) & \\
    			& & & & \sigma_{1,2}(r+1)
    		\end{pmatrix}.
    	\end{equation}
    	Note that the zeroth row defines the input $\sigma_{1,r+2} = w'$.
    \end{lemma}
    \begin{proof}
    	Write $m$ and the corresponding simple reflections in the form
    	\[
    	m = \begin{pmatrix}
    		m_{1,r+1} & m_{2,r+1} & \dots & m_{r-1,r+1} & m_{r,r+1}\\
    		& m_{1,r} & \dots & m_{r-2,r} & m_{r-1,r}\\
    		& &\ddots & \vdots & \vdots\\
    		& & & m_{13} & m_{23} \\
    		& & & & m_{12}
    	\end{pmatrix}, \quad
    	\begin{pmatrix}
    		s_{1} & s_{2} & \dots & s_{r-1} & s_{r}\\
    		& s_{2} & \dots & s_{r-1} & s_{r}\\
    		& &\ddots & \vdots & \vdots\\
    		& & & s_{r-1} & s_{r} \\
    		& & & & s_{r}
    	\end{pmatrix}.
    	\]
    	The evaluation of $\sigma$ goes from the top row of $m$ down to bottom, reading each row from right to left. At each entry we either multiply by the corresponding $s_j$ or skip it according to the colorings rules \cref{eq:instrutionsrules}. Note that after the $k$-th row the future permutations do not affect colors in positions $1,2,\dots,k$ since there are no reflections $s_1,\dots,s_{k}$ in the sequel. Hence, by induction it is enough to prove that permutations $\sigma_{k,r+1}$ for $k \in 1,2,\dots,r$ of the first row are determined by $\sigma_{1,r+1}$ alone. Indeed, by the colorings rules in \cref{eq:instrutionsrules}, 
    	\[
    		\sigma_{1,r+1} = (s_1^{\epsilon_1}\dots s_r^{\epsilon_r})w', \quad \epsilon_k \in \{0, 1\}.
    	\]
    	But there is a unique way to write $\sigma_{1,r+1}(w')^{-1}$ as a product of increasing simple reflections. Hence, $\sigma_{1,r+1}$ determines all $\epsilon_k$'s which in order determine all $\sigma_{k,r+1}$'s. Finally, $\sigma_{1,r+1}$ is determined by its action on colors $2,3,\dots,r+1$ which is the first row of $\eqref{eq:colortable}$ (after the zeroth one).
    \end{proof}
    
    \begin{example}
    	Let $m = \begin{pmatrix}
    		1 & 0\\
    		& 0
    	\end{pmatrix}$. Let the input be $w' = 1$. Let us use \Cref{lem:coloring} to write all possible colorings:
    	\begin{description}
    	    \item[output $w = 1$] $(1, s_2, s_2, 1)$, or $\begin{psmallmatrix}
        		\colorbox{red}{1} & \colorbox{green}{2} & \colorbox{blue}{\color{white}{3}}\\
        		& \colorbox{blue}{\color{white}{3}} & \colorbox{green}{2} \\
        		& & \colorbox{green}{2}
        	\end{psmallmatrix}$;\\
    	    \item[output $w = s_1$] $(s_2, s_2, s_2, 1)$, or $\begin{psmallmatrix}
        		\colorbox{red}{1} & \colorbox{green}{2} & \colorbox{blue}{\color{white}{3}}\\
        		& \colorbox{blue}{\color{white}{3}} & \colorbox{green}{2}\\
        		& & \colorbox{blue}{\color{white}{3}}
        	\end{psmallmatrix}$.
    	\end{description}
    \end{example}
    
    We cannot resist the temptation to merge $m$ with the coloring provided by \Cref{lem:coloring}. For convenience, we add the zeroth row to $m$ filled with dashes. We overlay $m$ and $\sigma$. If $m_{i,j}$ has color $k$, we write ${}_k m_{i,j}$. The example above becomes
    \[
    	m = \begin{pmatrix}
    		- & - & -\\
    		  & 1 & 0\\
    		  &   & 0 
    	\end{pmatrix}, \qquad 
    	\begin{pmatrix}
    		{}_{1}\colorbox{red}{--} & {}_{2}\colorbox{green}{--} & {}_{3}\colorbox{blue}{\color{white}{--}}\\
    		& {}_{3} \colorbox{blue}{\color{white}{1}} & {}_{2} \colorbox{green}{0}\\
    		& & {}_{2}\colorbox{green}{0}
    	\end{pmatrix} \text{ and } 
    	\begin{pmatrix}
    		{}_{1}\colorbox{red}{--} & {}_{2}\colorbox{green}{--} & {}_{3}\colorbox{blue}{\color{white}{--}}\\
    		& {}_{3} \colorbox{blue}{\color{white}{1}} & {}_{2} \colorbox{green}{0}\\
    		& & {}_{3} \colorbox{blue}{\color{white}{0}}
    	\end{pmatrix}.
    \]
    
    Let $m \in \N^N$ and $\lambda = (\lambda_1,\lambda_2,\dots,\lambda_{r+1}) \in \Lambda$ be a weight. We write $\lambda = \sum_i\Lambda_i \varpi_i$ as the sum of fundamental weights, so $\Lambda_i = \lambda_i - \lambda_{i+1}$. We define statistics $s_{i,j} = s_{i,j}(m, \lambda)$ by
    \[
        s_{i,j} = \Lambda_i + \sum_{k=j}^{r}m_{i+1,k+1} - \sum_{k=j}^{r+1}m_{i,k},
    \]
    
    \begin{definition}\label{def:lusztigdata}
        Let $\lambda$ be a weight. The finite set of $m \in \N^N$ such that $s_{i,j} \geq -1$ for all $(i,j) \in \Phi^+$ is called the \textit{Lusztig data corresponding to $\lambda+\rho$} and denoted by $\Lu(\lambda+\rho)$.
    \end{definition}
    
    Lusztig data is used to parametrize the Kashiwara crystal $\mathcal{B}(\lambda+\rho)$ for $\GL_{r+1}$ which explains the shift by $\rho$ in the notation. We exclusively use colorings for the Lusztig data. So we specialize from arbitrary strings $\N^N$ to the set of Lusztig data. 
    
    A \textit{colored Lusztig data} corresponding to $\lambda$ with input $w$ and output $w'$ is the pair $(m,\sigma)$ of a Lusztig datum $m \in \Lu(\lambda+\rho)$ and a $(m,w,w')$-coloring $\sigma$ that we visualize as a table of entries of $m$ colored by $\sigma$ using \Cref{lem:coloring} as discussed above. The set of all colored Lusztig data with input $w$ and output $w'$ is denoted by $\LLu(\lambda+\rho, w, w')$.
    
    This definition is not exactly useful since for a given Lusztig datum $m$ we need to find a coloring $\sigma$, make a coloring by \Cref{lem:coloring}, and only then merge the pattern with the coloring to get a colored Lusztig datum. Now we give a purely combinatorial definition of the colored Lusztig data that doesn't involve $(m,w,w')$-colorings at all.
    
    \begin{definition}\label{def:coloring}
        Let $\lambda \in \Lambda$ be a weight and $w,w' \in S_{r+1}$ be two permutations. A \textit{coloring} $\sigma$ of a Lusztig data $m \in \Lu(\lambda + \rho)$ with input $w'$ and output $w$ is an assignment of colors $1,2,\dots,r+1$ to each entry of $m$ according to the following procedure. 
        
        \begin{procedure}\label{procedure:coloring}
        Write $m \in \Lu(\lambda + \rho)$ in the form 
    	\[
        	m = \begin{pmatrix}
        		- & - & - & \dots & - & -\\
        		& m_{1,r+1} & m_{2,r+1} & \dots & m_{r-1,r+1} & m_{r,r+1}\\
        		& & m_{1,r} & \dots & m_{r-2,r} & m_{r-1,r}\\
        		& & &\ddots & \vdots & \vdots\\
        		& & & & m_{13} & m_{23} \\
        		& & & & & m_{12}
        	\end{pmatrix}.
    	\]
    	
    	 The zero-th row of dashes is given colors by the input permutation $w'$, that is, the zero-th row has colors $w'(1),w'(2),\dots,w'(r+1)$. Starting from the next row, we color entries from top to bottom, from right to left, using the following rules at each step. 
    	 
    	 You have a buffer color $e$. Every time you start a new row, you update $e$ to be the color to the top right most entry of the previous row. While in a row, $e$ will be updated as the procedure goes.  At each step you have a triangle of values and colors
    	 \[
        	 \begin{pmatrix}
            	{}_a A & {}_b B\\
            	& C
            \end{pmatrix}
        \]
        
        \begin{description}
            \item[If $C > 0$] Paint $C$ to the color $e$ and update $e = a$. 
            \item[If $C = 0$ and $a < b$] Paint $C$ to the color $a$ and don't update $e$
            \item[If $C= 0 $ and $a > b$] You can do one of the following 
                \begin{enumerate}
                    \item Paint $C$ to the color $e$ and update $e = a$
                    \item Paint $C$ to the color $a$ and don't update $e$
                \end{enumerate}
        \end{description}
        \end{procedure}
        
        Consider colors that are present in $k$-th row, but not in $(k+1)$-th row. Informally, they ''leave`` the table. They form a permutation which should be the output $w$.
    \end{definition}
    
    Note that the colorings from the definition above give all possible colorings by \Cref{lem:coloring}. Therefore, we give
    
    \begin{definition}
        Let $\lambda \in \Lambda$ be a weight and $w,w' \in S_{r+1}$ be two permutations. The set of \textit{colored Lusztig data} $\LLu(\lambda+\rho,w,w')$ is the set of all possible pairs of Lusztig data $m \in \Lu(\lambda+\rho)$ and colorings $\sigma \in \Sigma(m,w,w')$ given by \Cref{def:coloring} with input $w'$ and output $w$.
    \end{definition}
    
    \begin{remark}
        The coloring procedure itself is not an algorithm as it requires a choice in the case $C = 0$ and $a > b$. But it is the basis of an algorithm enumerating all colored Lusztig data as the leaves of a binary tree of all possible colorings.
    \end{remark}
    
    \begin{example}
    	Let $\lambda = (1, 0, 0)$ and let
    	\[
    	m = \begin{pmatrix}
    		- & - & - & -\\
    		& 1 & 0 & 0\\
    		&   & 0 & 1\\
    		&   &   & 0	
    	\end{pmatrix} \in \Lu(\lambda+\rho).
    	\]
    	
    	Here is the application of the coloring procedure to produce a coloring with input $w' = 1$ and output $w = \begin{psmallmatrix}
    	    1 & 2 & 3 & 4 \\
    	    1 & 2 & 4 & 3
    	\end{psmallmatrix}$.
    	\[
    	m = \begin{pmatrix}
    		{}_1\colorbox{red}{--} & {}_2\colorbox{green}{--} & {}_3\colorbox{blue}{\color{white}{--}} & _{4}\colorbox{magenta}{--} \\
    		& 1 & 0 & 0 \\
    		&   & 0 & 1 \\
    		&   &   & 0
    	\end{pmatrix} \xrightarrow{1}
    	\begin{pmatrix}
    		{}_1\colorbox{red}{--} & {}_2\colorbox{green}{--} & {}_3\colorbox{blue}{\color{white}{--}} & _{4}\colorbox{magenta}{--} \\
    		& 1 & 0 & {}_{3}\colorbox{blue}{\color{white}{0}} \\
    		&   & 0 & 1 \\
    		&   &   & 0
    	\end{pmatrix} \xrightarrow{2} 
    	\begin{pmatrix}
    		{}_1\colorbox{red}{--} & {}_2\colorbox{green}{--} & {}_3\colorbox{blue}{\color{white}{--}} & _{4}\colorbox{magenta}{--} \\
    		& 1 & {}_2\colorbox{green}{0} & {}_{3}\colorbox{blue}{\color{white}{0}} \\
    		&   & 0 & 1 \\
    		&   &   & 0
    	\end{pmatrix} \xrightarrow{3}
    	\]
    	\[
    	\xrightarrow{3}
    	\begin{pmatrix}
    		{}_1\colorbox{red}{--} & {}_2\colorbox{green}{--} & {}_3\colorbox{blue}{\color{white}{--}} & _{4}\colorbox{magenta}{--} \\
    		& {}_4\colorbox{magenta}{1} & {}_2\colorbox{green}{0} & {}_{3}\colorbox{blue}{\color{white}{0}} \\
    		&   & 0 & 1 \\
    		&   &   & 0
    	\end{pmatrix} \xrightarrow{4}
    	\begin{pmatrix}
    		{}_1\colorbox{red}{--} & {}_2\colorbox{green}{--} & {}_3\colorbox{blue}{\color{white}{--}} & _{4}\colorbox{magenta}{--} \\
    		& {}_4\colorbox{magenta}{1} & {}_2\colorbox{green}{0} & {}_{3}\colorbox{blue}{\color{white}{0}} \\
    		&   & 0 & {}_3\colorbox{blue}{\color{white}{1}} \\
    		&   &   & 0
    	\end{pmatrix} \xrightarrow{5} 
    	\begin{pmatrix}
    		{}_1\colorbox{red}{--} & {}_2\colorbox{green}{--} & {}_3\colorbox{blue}{--} & _{4}\colorbox{magenta}{--} \\
    		& {}_4\colorbox{magenta}{1} & {}_2\colorbox{green}{0} & {}_{3}\colorbox{blue}{\color{white}{0}} \\
    		&   & {}_4 \colorbox{magenta}{0} & {}_3\colorbox{blue}{\color{white}{1}} \\
    		&   &   & 0
    	\end{pmatrix} \xrightarrow{6} 
    	\]
    	\[
    	\xrightarrow{6}
    	\begin{pmatrix}
    		{}_1\colorbox{red}{--} & {}_2\colorbox{green}{--} & {}_3\colorbox{blue}{\color{white}{--}} & _{4}\colorbox{magenta}{--} \\
    		& {}_4\colorbox{magenta}{1} & {}_2\colorbox{green}{0} & {}_{3}\colorbox{blue}{\color{white}{0}} \\
    		&   & {}_4 \colorbox{magenta}{0} & {}_3\colorbox{blue}{\color{white}{1}} \\
    		&   &   & {}_3\colorbox{blue}{\color{white}{0}}
    	\end{pmatrix}.
    	\]
    	
    	At the beginning, the buffer color $e = 4$.
    	\begin{enumerate}
    	    \item[Step 1:] $C = 0$ and $3 = a < b = 4$. Paint $C = 0$ to the color $3$ (blue)
    	    \item[Step 2:] $C = 0$ and $2 = a < b = 3$. Paint $C = 0$ to the color $2$ (green)
    	    \item[Step 3:] $C = 1$. Paint $C = 1$ to the color $4$ and set $e = 1$ (red)
    	    \item[New row:] Set $e = 3$ (blue)
    	    \item[Step 4:] $C = 1$. Paint $C = 1$ to the color $e = 3$ (blue)
    	    \item[Step 5:] $C = 0$ and $4 = a > b = 2$. We have a choice. We choose to paint $C = 0$ to the color $a = 4$. We do not update color $e$.
    	    \item[New row:] Set $e = 3$ (blue)
    	    \item[Step 6:] $C = 0$ and $4 = a > b = 3$. We have a choice. We choose to paint $C = 0$ to the color $e = 3$ (blue)
    	\end{enumerate}
    	
    	All possible colorings with the input $\sigma' = 1$ are listed below. 
    	
    	\[
    		\left\{\begin{pmatrix}
    			{}_1\colorbox{red}{--} & {}_2\colorbox{green}{--} & {}_3\colorbox{blue}{\color{white}{--}} & _{4}\colorbox{magenta}{--} \\
    			& {}_4\colorbox{magenta}{1} & {}_2\colorbox{green}{0} & {}_{3}\colorbox{blue}{\color{white}{0}} \\
    			&   & {}_4 \colorbox{magenta}{0} & {}_3\colorbox{blue}{\color{white}{1}} \\
    			&   &   & {}_4\colorbox{magenta}{0}
    		\end{pmatrix}, 
    		\begin{pmatrix}
    			{}_1\colorbox{red}{--} & {}_2\colorbox{green}{--} & {}_3\colorbox{blue}{\color{white}{--}} & _{4}\colorbox{magenta}{--} \\
    			& {}_4\colorbox{magenta}{1} & {}_2\colorbox{green}{0} & {}_{3}\colorbox{blue}{\color{white}{0}} \\
    			&   & {}_4 \colorbox{magenta}{0} & {}_3\colorbox{blue}{\color{white}{1}} \\
    			&   &   & {}_3\colorbox{blue}{\color{white}{0}}
    		\end{pmatrix},
    		\begin{pmatrix}
    			{}_1\colorbox{red}{--} & {}_2\colorbox{green}{--} & {}_3\colorbox{blue}{\color{white}{--}} & _{4}\colorbox{magenta}{--} \\
    			& {}_4\colorbox{magenta}{1} & {}_2\colorbox{green}{0} & {}_{3}\colorbox{blue}{\color{white}{0}} \\
    			&   & {}_2 \colorbox{green}{0} & {}_3\colorbox{blue}{\color{white}{1}} \\
    			&   &   & {}_2\colorbox{green}{0}
    		\end{pmatrix}\right\}.
    	\]
    \end{example}
    
    Let $\textbf{m} = (m,\sigma) \in \LLu(\lambda+\rho)$ be a colored Lusztig datum. A weight $z^\textbf{m}$ is defined by
    \[
        z^\textbf{m} = \prod_{\alpha \in \Phi^+}z^{m_\alpha \alpha} = \prod_{i < j}(z_i/z_j)^{m_{ij}}.
    \]
    
    Let 
    \[
    	\begin{matrix}
    		{}_aA & {}_bB & \\
    		& {}_cC & {}_d D 
    	\end{matrix}
    \] 
    be a block of values and colors in the datum $\textbf{m}$, where $C$ is the position of $m_{i,j}$. If $C$ is the right-most entry in $m$, for convenience, we assume $D > 0$. Then the contribution $\G(\textbf{m}, \alpha)$ at $\alpha = (i,j) \in \Phi^+$ is defined by 
    \begin{equation}\label{eq:lusztigcontribution}
    	\resizebox{\textwidth}{!}{$\displaystyle
        \G\left(\begin{matrix}
    		{}_aA & {}_bB &\\
    	       	  & {}_cC & {}_d D 
    	\end{matrix}\right) \eqqcolon \begin{cases}
    	    \begin{cases}
    			g(r_{i,j}, s_{i,j}), &\text{if $C > 0$}\\
    			1, &\text{if $C = 0$ and $a < b$ and $s_{ij} \geq 0$}\\
    			0, &\text{if $C = 0$ and $a < b$ and $s_{ij} = -1$}\\
    			1-q^{-1}, &\text{if $C = 0$ and $a > b$, $c \neq a$ and $s_{ij}  \geq 0$}\\
    			-q^{-1}, &\text{if $C = 0$ and $a > b$, $c \neq a$ and $s_{ij} = -1$}\\
    			q^{-1}, &\text{if $C = 0$ and $a > b$, $c = a$ and $s_{ij} \geq 0$}\\
    			q^{-1}, &\text{if $C = 0$ and $a > b$, $c = a$ and $s_{ij} = -1$}
    		\end{cases}, &\parbox[t]{.2\textwidth}{if $D > 0$ or \\if $D = 0$ and $d \neq b$}\\
    		\begin{cases}
    			g(r_{i,j}, 0), &\text{if $C > 0$}\\
    			1, &\text{if $C = 0$ and $a < b$}\\
    			1-q^{-1}, &\text{if $C = 0$ and $a > b$, $c \neq a$}\\
    			q^{-1}, &\text{if $C = 0$ and $a > b$, $c = a$}
    		\end{cases}, &\text{if $D = 0$ and $d = b$}.
    	\end{cases}.$}
    \end{equation}
    
    Note that the weight of a colored Lusztig datum is zero if there is any root $(i, j) \in \Phi^+$ with $C = m_{i,j} = 0$, $s_{i,j} = -1$ and $a < b$ when $D > 0$ or ($D = 0$ and $d \neq b$). It motivates the following
    
    \begin{definition}
        A colored Lusztig datum $\textbf{m} \in \LLu(\lambda+\rho,w,w')$ is called \textit{strict} if there are no blocks of values and colors
        \[
        	 \begin{pmatrix}
            	{}_a A & {}_b B &\\
    	       	  & {}_c C & {}_d D
            \end{pmatrix}
        \]
        with $C = m_{i,j} = 0$, $s_{ij} = -1$, and $a < b$ when $D > 0$ or ($D = 0$ and $d \neq b$). Denote the set of strict colored Lusztig data by $\SLLu(\lambda+\rho,w,w')$. 
    \end{definition}
    
    In the $n$-metaplectic case, $g(r_{\alpha}, 0) = 0$ if $r_{\alpha} \not\equiv 0 \pmod{n}$. Analyzing the weight above, we see that it is zero if $r_\alpha \not\equiv 0 \pmod{n}$, $s_{ij} = 0$, when $D > 0$ or ($D = 0$ and $d \neq b$); or if $r_{i,j} \not\equiv 0 \pmod{n}$ when $D = 0$ and $d = b$. It motivates the following
    
    \begin{definition}
        A colored Lusztig datum $\textbf{m} \in \LLu(\lambda+\rho,w,w')$ is called ($n$-)\textit{superstrict} if $m$ is strict and $r_{\alpha} \equiv 0 \pmod{n}$ for all $\alpha \in \Phi^+$ if $s_{i,j} = 0$ when $D > 0$ or ($D = 0$ and $d \neq b$); or if $D = 0$ and $d = b$. Denote the set of strict colored Lusztig data by $\SSLLu(\lambda+\rho,w,w')$. 
    \end{definition}
    
    Prefix ``super'' comes from the supersymmetric lattice models that we will introduce later. 
    
    \subsection{Colored Gelfand-Tsetlin Patterns}
	We show that the colored Lusztig data is is in a weight-preserving bijection with the colored Gelfand-Tsetlin patterns which we define momentarily.
    
    A Gelfand-Tsetlin pattern (or GT-pattern for short) with the top row $(a_{11}, a_{12}, \dots, a_{1n}) \in \N^n$ is a triangular array of integers 
    \[
    	\left\{\begin{matrix}
    
        a_{11} & & a_{12} & & \dots & & a_{1n}\\
    
        & a_{22} & & \dots & & a_{2n} &\\
        & & \ddots & & \udots & & \\
        & & & a_{nn} & & &
        
        \end{matrix}\right\}
    \]
    such that the betweenness condition $a_{i,j} \geq a_{i+1,j+1} \geq a_{i,j+1}$ is satisfied for all $i,j$. In other words, each entry in the pattern lies between two entries above it. Denote by $\text{GT}(a_{11},a_{12},\dots,a_{1n})$ the set of GT-patterns with the top row $(a_{11},a_{12},\dots,a_{1n})$.
    
    Let $\mu$ be a dominant weight, that is, a partition. For our convenience, we parametrize $\GT(\mu)$ by solutions $(a_{1,2}, a_{1,3}, \dots, a_{r,r+1}) \in \N^N$ of the following inequalities: for each $(i, j) \in \Phi^+$, $a_{i,j} \geq 0$, and
    \[
        a_{i,j+1} \geq a_{i,j} \geq a_{i+1,j+1},
    \]
    where we denote $a_{i,r+2} = \mu_{i}$ for each $i = 1, 2, \dots, r+1$. This is how it looks:
    \[
    	\left\{\begin{matrix}
    
        \mu_1 &  & \mu_2 & & \dots & & \mu_{r} & & \mu_{r+1}\\
    	& a_{1,r+1} & & a_{2,r+1} & \dots & a_{r-1,r+1} & & a_{r,r+1} & \\
    	&  & a_{1,r} & & \dots & & a_{r-1,r} & & \\
    	&  &  \ddots&  & &  &\udots &\\
    	&  &  & a_{1,3} & & a_{2,3} & & \\
    	&  & & & a_{1,2} & & & 
        
        \end{matrix}\right\}
    \]
    
    A weight of a GT-pattern $T$ is defined by 
    \[
        z^T = \prod_{\alpha \in \Phi^+}z^{(a_{i,j+1}-a_{ij}) \alpha} = \prod_{i < j}(z_i/z_j)^{(a_{i,j+1}-a_{ij})}.
    \]
    
    Let $d_k = d_k(T)$ be the sum of entries in $k$-th row of a GT-pattern $T$, and $d_{r+2} = 0$. Then the weight $z^T$ is also given by 
    \[
        z^{T} = z^{\lambda+\rho}\prod_{i=1}^{r+1}z_i^{d_{i+1}-d_i}.
    \]
    
    Indeed, we get the power of $z_k$ in $z^T = \prod_{i < j}(z_i/z_j)^{a_{i,j+1}-a_{ij}}$ as follows: 
    \[
        \sum_{j=k+1}^{r+1}(a_{k,j+1}-a_{k,j}) -\sum_{i=1}^{k-1}(a_{i,k+1}-a_{i,k}) = d_{k+1} - d_k + (\lambda_k + r+1 - k). 
    \]
    
    Hence, $z^T = z^{\lambda+\rho}\prod_{i=1}^{r+1}z_i^{d_{i+1}-d_i}$, as required.
    
    We now show that GT-patterns with the top row $\lambda + \rho$ are in weight-preserving bijection with Lusztig data corresponding to $\lambda + \rho$. 
    
    \begin{lemma}\label{lem:bijectionlusztiggt}
        There is a weight-preserving bijection between GT-patterns with the top row $\lambda + \rho$ and colored Lusztig data corresponding to $\lambda + \rho$. The map $f \colon \Lu(\lambda + \rho) \to \GT(\lambda + \rho)$ defined by $f(m) = T = (a_{1,2},\dots,a_{r,r+1}) \in \N^N$ with $a_{i,j} = a_{i,j+1} - m_{i,j}$ for all $(i,j) \in \Phi^+$ is a weight-preserving bijection, that is, $z^m = z^T$. 
    \end{lemma}
    \begin{proof}
        First we show that $f(m)$ is a GT-pattern, that is, $a_{i,j+1} \geq a_{i,j} \geq a_{i+1,j+1}$. 
        
        Note that $a_{i,r+2} = \lambda_i + r+1 - i$ since the top row is $\lambda + \rho$. Thus, by induction, 
        \[
            a_{i,j} = a_{i,r+2} - \sum_{k=j}^{r+1}m_{i,k} = \lambda_i + r+1 - i - \sum_{k=j}^{r+1}m_{i,k}.
        \]
        
        Recall that $s_{i,j} = \lambda_i - \lambda_{i+1} + \sum_{k=j}^{r}m_{i+1,k+1} - \sum_{k=j}^{r+1}m_{i,k}$ and so $s_{i,j} = a_{i,j} - a_{i+1,j+1} - 1$. Since $m \in \Lu(\lambda + \rho)$, we have $s_{i,j} \geq -1$ for all $(i,j) \in \Phi^+$. Thus, we get $s_{i,j} = a_{i,j} - a_{i+1,j+1} -1  \geq -1$, or $a_{i,j} \geq a_{i+1,j+1}$. On the other hand, since $a_{i,j} = a_{i,j+1} - m_{i,j}$ and $m_{i,j} \geq 0$, we get $a_{i,j+1} \geq a_{i,j}$. Thus, $f(m)$ is indeed a GT-pattern. The same argument shows that the inverse map $g \colon \GT(\lambda+\rho) \to \Lu(\lambda+\rho)$ defined by $g(T) = (m_{1,2}, \dots, m_{r,r+1}) \in \N^N$ with $m_{i,j} = a_{i,j+1} - a_{i,j}$ is the inverse to $f$, and so $f$ is a bijection. Since $m_{ij} = a_{i,j+1}-a_{i,j}$, the bijection is weight-preserving, that is, $z^m = z^T$.
    \end{proof}
    
    \begin{remark}
    	More generally, for a dominant weight $\lambda$, the map given by the same formula yeilds a bijection $\Lu(\lambda) \to \GT(\lambda)$ between two parametrizations of the crystal $\mathcal{B}(\lambda)$. We won't need this result in the sequel.
    \end{remark}
    
    Let $\textbf{m} \in \LLu(\lambda + \rho,w,w')$ be a colored Lusztig data. By the bijection above, we can define the coloring of a GT-patterns by coloring each entry $a_{i,j}$ in the same color as $m_{i,j}$ under the bijection. It defines colored GT-patterns $\GGT(\lambda+\rho,w,w')$. 
    
    For convenience, we rewrite the step of coloring \Cref{procedure:coloring} for the GT-patterns. The coloring of the top row is given by the input $w'$. Starting from the next row, we color entries from top to bottom, from right to left, using the following rules at each step.
    \[
    \left\{\begin{matrix}
    	{}_a A &   &  {}_b B\\
     	       & C & 
    \end{matrix}\right\} \to 
    \begin{cases}
        \left\{\begin{matrix}
    		{}_a A & & {}_b B\\
    		& {}_e C &
    	\end{matrix}\right\}, &\text{if $C > A$}\\
    	\left\{\begin{matrix}
    		{}_a A & & {}_b B\\
    		& {}_a C &
    	\end{matrix}\right\}, &\text{if $C = A$ and $a < b$}\\
    	\left\{\begin{matrix}
    		{}_a A & & {}_b B\\
    		& {}_a C &
    	\end{matrix}\right\} \text{ or } 
    	\left\{\begin{matrix}
    		{}_a A & & {}_b B\\
    		& {}_e C &
    	\end{matrix}\right\}, &\text{if $C = A$ and $a > b$}
    	
    \end{cases}.
    \]
    
    Let $\textbf{T} \in \GGT(\lambda+\rho,w,w')$. Let 
    \[
    	\begin{matrix}
    		{}_a A &  & {}_b B &\\
    		& {}_c C &   & {}_d D 
    	\end{matrix}
    \] 
    be a block of values and colors in the pattern $\textbf{T}$, where $C$ is the position of $a_{i,j}$. If $C$ is the right-most entry in $T$, for convenience we assume $D < B$. Then the contribution $\G(\textbf{T}, \alpha)$ at $\alpha = (i,j) \in \Phi^+$ is defined by 
    
    \begin{equation}
    	\resizebox{\textwidth}{!}{$\displaystyle
        \G\left(\begin{matrix}
    		{}_a A &  & {}_b B &\\
    		& {}_c C &   & {}_d D 
    	\end{matrix}\right) \eqqcolon \begin{cases}	
    		\begin{cases}
    			g(r_{i,j}, s_{i,j}), &\text{if $C < A$}\\
    			1, &\text{if $C = A > B$ and $a < b$}\\
    			0, &\text{if $C = A = B$ and $a < b$}\\
    			1-q^{-1}, &\text{if $C = A > B$ and $a > b$ and $c = b$}\\
    			-q^{-1}, &\text{if $C = A = B$ and $a > b$ and $c = b$}\\
    			q^{-1}, &\text{if $C = A > B$ and $a > b$ and $c = a$}\\
    			q^{-1}, &\text{if $C = A = B$ and $a > b$ and $c = a$}
    		\end{cases}, &\parbox[t]{.2\textwidth}{if $D < B$ or \\if $D = B$ and $d \neq b$}\\
    		\begin{cases}
    			g(r_{i,j}, 0), &\text{if $C < A$}\\
    			1, &\text{if $C = A$}\\
    			1-q^{-1}, &\text{if $C = A$ and $a < b$ and $c = b$}\\
    			q^{-1}, &\text{if $C = A$ and $a > b$ and $c = a$}
    		\end{cases}, &\text{if $D = B$ and $d = b$}
    	\end{cases}.$}
    \end{equation}
    
    The statistics $r_{ij}$ and $s_{ij}$ and the Gauss sums $g(r_{ij},s_{ij})$ are defined in \Cref{section:calculation}. We rewrite them in terms of colored GT-patterns: $s_{ij} = a_{ij} - a_{i+1,j+1} - 1$ and $r_{ij} = \sum_{k \leq i}(a_{k,j+1}-a_{kj})$. 
    
    A colored GT-pattern $T$ is called \textit{strict} if no triangle of values and colors 
    \[\left\{\begin{matrix}
    	{}_a A &  & {}_b A &\\
    		& A &   & {}_d D 
    \end{matrix}\right\}\]
    with $a < b$ when $D < A$ or ($D = A$ and $d \neq b$) is present in $T$. Denote the set of strict colored GT-patterns by $\SGGT(\lambda+\rho,w,w')$. A colored GT-pattern $T$ is called ($n$-)\textit{superstrict} if $T$ is strict and $r_{ij} \equiv 0 \pmod{n}$ for all entries $C$ with $C < A < B$ when $D < B$ or ($D = B$ and $d \neq b$), or when $D = B$ and $d = b$ in the notation above. The same map from \Cref{lem:bijectionlusztiggt} gives the bijection between strict Lusztig data and strict GT-patterns; and between $n$-superstrict Lusztig data and $n$-superstrict GT-patterns. 
    
    When computing examples, GT-patterns are especially convenient because of the following observation. A term $a_{ij}$ with $s_{ij} = -1$ will be \textit{right-leaning}, that is, $a_{ij} = a_{i+1,j+1}$. Similarly, an entry with $m_{i,j} = 0$, will be \textit{left-leaning}, that is, $a_{ij} = a_{i,j+1}$. In \cite{BBF11}, such entries were decorated by boxes and circles.
    
    \subsection{Colored Lattice Models}
    We show that the colored data introduced above is in a weight-preserving bijection with the supersymmetric lattice models dual to the ones from \cite{BBBG19, BBBG20}. We work with the delta version of the models while in \cite{BBBG19, BBBG20} the gamma version is used. See the details in Section 8 of \cite{BBF11}.
    
    \subsubsection{Non-metaplectic case $n=1$}
    
    We introduce the colored lattice model $\mathfrak{S}(\mu,w,w')$ that depends on a partition $\mu$ of length $r+1$ and two permutations $w,w' \in S_{r+1}$. 
    
    The model is a rectangular grid consisting of $\mu_1 + 1$ columns numbered from $\mu_1$ to $0$ from left to right, and $n$ rows numbered from $r+1$ to $1$ from top to bottom. We launch $r+1$ colored paths of $r+1$ different colors that can go only down and left. Any number of paths can occupy a given vertical edges, but only one color can occupy a given horizontal edge. More formally, we allow only configurations from \Cref{fig:configurations} in our model. The edges with no colors are marked by plus signs. See Sections 6-7 of \cite{BBBG19} for details. 
    
    \newcommand{\arrowdiagram}[4]{
    \begin{tikzpicture}[anchor=base, baseline]
      \draw[line width=1.0pt ] (-1,0) node[left]{$#1$} -- (1,0) node[right]{$#3$};
      \draw[line width=5pt] (0,1) node[above]{$#2$} -- (0,-1) node[below]{$#4$};
    \end{tikzpicture}
    }
    
    \begin{figure}[ht]
    \begin{equation*}
        \resizebox{\textwidth}{!}{$\displaystyle
        \begin{array}[t]{|c|c|c|}\hline
            \arrowdiagram{+}{\Sigma}{+}{\Sigma} & \arrowdiagram{c}{\Sigma}{c}{\Sigma} & \arrowdiagram{c}{\Sigma}{+}{\Sigma \setminus \{c\}} \\
            y_i(-q^{-1})^{|\Sigma|} & (q^{-1})^{|\Sigma \cap [1,c-1]|} & (-q^{-1})^{|\Sigma \cap [c+1,r+1]|}(q^{-1})^{|\Sigma \cap [1,c-1]|}
            \\ \hline
            \arrowdiagram{+}{\Sigma}{c}{\Sigma \cup \{c\}} & \arrowdiagram{c}{\Sigma}{d}{\Sigma \cup \{d\} \setminus \{c\}} & \arrowdiagram{d}{\Sigma}{c}{\Sigma \cup \{c\} \setminus \{d\}}
            \\
            y_i(1-q^{-1})(-q^{-1})^{|\Sigma \cap [1,c-1]|} & (1-q^{-1})(-q^{-1})^{|\Sigma \cap [d-1,c-1]|}(q^{-1})^{|\Sigma \cap [1,d-1]|} & \text{not admissible!}
            \\ \hline
          \end{array} $
            }
    \end{equation*}
    \caption{The admissible configurations and Boltzmann weights on $i$-th row. Here $\Sigma$ is a set of colors passing through the vertical edge. The empty set corresponds to the plus sign. We assume $c > d$.}\label{fig:configurations}
    \end{figure}
    
    Weight $\mu$ and permutations $w,w'$ specify the boundary conditions: paths are launched at columns $\mu_1,\mu_2,\dots,\mu_{r+1}$ and have colors $w'(1), w'(2), \dots, w'(r+1)$ in this order. Moreover, the left boundary condition specifies what colors can leave at each row. The order of leaving colors is $w(1), w(2), \dots, w(r+1)$. See \Cref{fig:latticeexample} for an example. 
    
    \begin{figure}
        \centering
        \begin{tikzpicture}[xscale=0.75, yscale=0.5, font=\small]
            \grid{3}{4}
            
            \draw (1, 8) node[line width=1.5pt, red, spin] {$1$};
            \draw (5, 8) node[line width=1.5pt, cyan, spin] {$\scriptstyle 2,3$};
            \draw (7, 8) node[line width=1.5pt, magenta, spin] {$4$};
            
            \draw (0, 7) node[line width=1.5pt, blue, spin] {$3$};
            \draw (0, 5) node[line width=1.5pt, red, spin] {$1$};
            \draw (0, 3) node[line width=1.5pt, magenta, spin] {$4$};
            \draw (0, 1) node[line width=1.5pt, green, spin] {$2$};
        \end{tikzpicture} \quad \begin{tikzpicture}[xscale=0.75, yscale=0.5, font=\small]
            \newcommand\eps{0.15}
            \grid{3}{4}
            
            \draw [red, path] 
            (1,8) node {$1$} to
            (1,6) node {$1$} to [out=-90,in=0,looseness=2.0]
            (0, 5) node {$1$};
            
            \draw [green, path] 
            (5,8) to
            (5,6) node {$2$} to
            (5,4) to
            (5,2) node {$2$} to [out=-90,in=0,looseness=2.0]
            (4, 1) node {$2$} to
            (2, 1) node {$2$} to
            (0, 1) node {$2$};
            
            \draw [blue, path]
            (5-\eps, 8) to [out=-90, in=0, looseness=2.0]
            (4, 7) node {$3$} to 
            (2, 7) node {$3$} to
            (0, 7) node {$3$};
            
            \draw [magenta, path]
            (7,8) node {$4$} to 
            (7, 6) node {$4$} to [out=-90, in=0, looseness=2.0] 
            (6, 5) node {$4$} to [out=180, in=90, looseness=2.0] 
            (5+\eps, 4) to [out=-90, in=0, looseness=2.0]
            (4, 3) node {$4$} to 
            (2, 3) node {$4$} to
            (0, 3) node {$4$};
            
            \draw (5, 8) node[line width=1.5pt, cyan, spin] {$\scriptstyle 2,3$};
            
            \draw (5, 4) node[line width=1.5pt, brown, spin] {$\scriptstyle 2,4$};
        \end{tikzpicture}
        
        \caption{On the left are the boundary conditions of the model $\mathfrak{S}(\mu,w,w')$ for $\mu = (3,1,1,0)$, $w = 1$, $w' = \begin{psmallmatrix}1 & 2 & 3 & 4 \\ 3 & 1 & 4 & 2\end{psmallmatrix}$; on the right an example of a state in this model.}
        \label{fig:latticeexample}
    \end{figure}
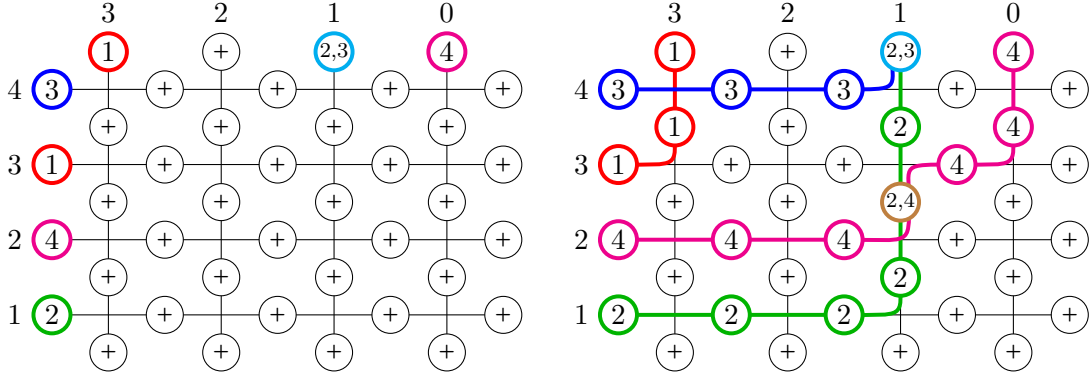
    
    \begin{remark}
        Note that weights we introduce in our model are dual to the weights from Figure 12 of \cite{BBBG19} in the same way as the delta weights are dual to the gamma weights in \cite{BBF11}. Hence, we can call the model in \cite{BBBG19} the \textit{colored gamma ice} and our model the \textit{colored delta ice}.
    \end{remark}
    
    Let $S \in \mathfrak{S}(\mu,w,w')$ be a state in the model. We define a weight of $S$ as the product of Boltzmann weights of all vertices in $S$:
    \[
        \wt(S) = \prod_{v \in S}\wt(v),
    \]
    where Boltzmann weights $\wt(v)$ are given in \Cref{fig:configurations}. 
    
    \begin{lemma}\label{lem:gtstatesbijection}
        There is a weight-preserving bijection between strict colored $GT$-patterns $\SGGT(\lambda+\rho,w,w')$ and colored lattice states $\mathfrak{S}(\lambda+\rho,w,w')$. The weight-preserving property here means that 
        \[
            \G(T) = \prod_{\alpha \in \Phi^+}\G(T,\alpha) = z^{\lambda+\rho}\prod_{v \in S}\wt(v) = z^{\lambda+\rho}\wt(S),
        \]
        where $y_i = z_i^{-1}$.
    \end{lemma}
    \begin{proof}
        We first notice that a state in the model is uniquely determined by its vertical edges. Indeed, admissible configurations allow paths to move only down and left. The number of rows is exactly the number of paths. Then the vertical edges prescribe the positions at which each colored path must descent and at each row one path leaves the state.
        
        A colored GT-pattern uniquely specifies the positions and colors of paths on vertical edges. The value corresponds to the position and the color corresponds to the color of a path. Conversely, each state in the lattice models gives a unique colored GT-pattern. 
        
        For uniqueness, notice that the only obstruction is when we have the following triangle of entries and colors in a colored GT-pattern:
        \[
            \left\{\begin{matrix}
            	{}_a A &   &  {}_b A\\
             	       & A & 
            \end{matrix}\right\}.
        \]
        By strictness of GT-patterns, we have $a > b$, hence, there is no ambiguity.
        
        For the weight-preserving property we consider separately the contribution of $z$ to the weight and the contribution of $q$. Note that the only vertices that give $z$-contributions are vertices with an empty horizontal edge. Their number is exactly the difference between two elements in a GT-pattern from two consequent rows. Hence, all together, we get exactly the difference of rows.
        
        For $q$ contribution, we just follow how a GT-pattern is mapped to a lattice state. We consider an example of the first type of a vertex in \Cref{fig:configurations}. Consider part of a GT-pattern with the following configuration where $a > b$. 
        \[
            \left\{
                \begin{matrix}
                    {}_{c} A & & {}_{c_1} B & \dots & {}_{c_{k}} B &       & {}_{c_{k}} B\\
                      & {}_{d_1} B & & {}_{d_2} B         & \dots & {}_{d_{k}} B
                \end{matrix}
            \right\}
        \]
        
        Then it maps exactly to the first type of a vertex from \Cref{fig:configurations}. The only possible configurations for weights to be strict is when $c_1 > c_2 > \dots c_{k}$ and $d_1 > d_2 > \dots > d_k$ due to the strictness condition. But then it forces $d_k = c_k$ for all $k$. The weight of such configuration is $(-q^{-1})^k$ which is exactly the $q$ contribution of the corresponding vertex in the lattice model. 
    
        Other vertices types are similar. 
    \end{proof}
    
    \subsubsection{Metaplectic case}
    Now we introduce the supersymmetric lattice model $\mathfrak{S}^n(\mu,w,w')$ that depends a partition $\mu$ of length $r+1$, two permutations $w,w' \in S_{r+1}$, and an integer $n = 1, 2, \dots$. The case $n=1$ corresponds to the non-metaplectic model from the previous section. The supersymmetric model is the dual lattice model to the one considered in \cite{BBBG20}, so we use their vocabulary. In particular, they explain that in physics, the prefix s is used to imply that scolor is the supersymmetry partner of color. Scolors are elements on $\Z/n\Z$. 
    
    The model is a rectangular grid consisting of $\mu_1 + 1$ columns numbered from $\mu_1$ to $0$ from left to right. We launch $r+1$ colored paths of $r+1$ different colors that can go only down and left. We also launch $r+1$ scolored paths of different scolors that can go only down and right. Any number of colors or scolors can occupy a given vertical edges, but only one color or scolor can occupy a given horizontal edge. Moreover, colors and scolors should use the same vertical edges. More formally, we allow only configurations from \Cref{fig:superconfigurations} in our model. The edges with no color are marked by plus signs. See \cite{BBBG20} for details. 
    
    For convenience, we write $c,d$ for colors, and $\overline{j}, \overline{f}$ for scolors. In other words, $\overline{j}, \overline{f} \in \Z/n\Z$. 
    
    \begin{figure}[ht]
    \begin{equation*}
        \resizebox{\textwidth}{!}{$\displaystyle
        \begin{array}[t]{|c|c|c|c|c|}\hline
            \arrowdiagram{\overline{f}}{\Sigma}{\overline{f}}{\Sigma} & \arrowdiagram{c}{\Sigma}{c}{\Sigma} & \arrowdiagram{c}{\Sigma}{\overline{j}}{\Sigma \setminus \{c\}} \\ 
            y_i g(\overline{f}-\overline{j},-1)^{|\Sigma|} & (q^{-1})^{|\Sigma \cap [1,c-1]|} & (-q^{-1})^{|\Sigma \cap [c+1,r+1]|}(q^{-1})^{|\Sigma \cap [1,c-1]|}\\  \hline
            \arrowdiagram{\overline{j}}{\Sigma}{c}{\Sigma \cup \{c\}} & \arrowdiagram{c}{\Sigma}{d}{\Sigma \cup \{d\} \setminus \{c\}} & \arrowdiagram{d}{\Sigma}{c}{\Sigma \cup \{d\} \setminus \{c\}}
            \\
            y_i(1-q^{-1})(-q^{-1})^{|\Sigma \cap [1,c-1]|} & (1-q^{-1})(-q^{-1})^{|\Sigma \cap [d-1,c-1]|}(q^{-1})^{|\Sigma \cap [1,d-1]|} & \text{not admissible!}
            \\ \hline
          \end{array}$
          }
    \end{equation*}
    \caption{The admissible configurations and Boltzmann weights. We consider $i$-th row and $j$-th column. Here $\Sigma$ is a set of colors passing through the vertical edge. The empty set corresponds to the plus sign. In the second row are corresponding weights on the row labeled by $i$. We assume $c > d$.}\label{fig:superconfigurations}
    \end{figure}
    
    Note that the only admissible configurations are the ones where a path with scolor $\overline{e}$ descends on columns $j \equiv e \pmod{n}$. It will be useful because the Gauss sum $g(a,0) = 0$ when $a \not\equiv 0 \pmod{n}$. The possible configurations will rule out the states with zero contributions. 
    
    \begin{lemma}
        The Boltzmann weight of a state in the supersymmetric model equals to the weight of the corresponding $n$-superstrict colored GT-pattern under the bijection given in \Cref{lem:gtstatesbijection}.
    \end{lemma}
    \begin{proof}
        We build on \Cref{lem:gtstatesbijection}. We already have a bijection between strict GT-patterns and colored states. Now we refine it to get a bijection between superstrict GT-patterns and supersymmetric states. Indeed, the supercolored paths give exactly the restriction on $r_{ij}$ which are given in the definition of the superstrict GT-patterns. 
        
        For the weights, we notice that $g(r_{ij}, 0)$ is taken care by the limitation of admissible states (or superstrict GT-patterns). The only thing we need to change now is to replace $(-q)^{-1}$ with $g(\overline{f}-\overline{j},-1)$ which is what we do in \Cref{fig:superconfigurations}. 
    \end{proof}
	
	\section{Computing metaplectic Iwahori Whittaker functions}\label{section:calculation}
	In this section we compute Iwahori Whittaker functions for the metaplectic covers of the general linear group $G = \GL_{r+1}$. We inherit notation from \Cref{section:coloreddata}: basis $e_i$, weights $\Lambda = X_*(T) \cong \Z^{r+1}$, roots $\Phi = \{(i,j) \in \Z^2 \mid 1 \leq i < j \leq r+1 \}$, positive roots $\Phi^+ = \{(i,j) \in \Z^2 \mid 1 \leq i < j \leq r+1 \}$, and the long word decomposition $\textbf{i} = \Delta$.
	
	Let $\widetilde{G}$ be a $n$-fold metaplecic cover of $G$ from \Cref{section:metaplecticgroups} coresponding to the bilinear form $B$ defined by $B(e_i,e_j) = \delta_{i,j}$. Then $l_\alpha = Q(\alpha^\vee) = B(\alpha^\vee, \alpha^\vee)/2 = 1$ for any $\alpha \in \Phi^+$. At the end of the section we consider other metaplectic covers.
	
	The Iwahori Whittaker functions are determined by \eqref{eq:miw}:
    \[
        \phi_w(\lambda,w'; z) = \int_{U} f_w(uw')\psi_\lambda(u)\,du.
    \]
    
    Recall that $f_w(uw') = f(u)$ for $u \in U \cap Bw_0wJ(w')^{-1}$ and zero otherwise. Then by \Cref{thm:celldecomposition}, we can rewrite values $\phi_w(\lambda,w'; z)$ as
    \[
        \phi_w(\lambda,w'; z) = \int_{U} f_w(uw')\psi_\lambda(u)\,du = \sum_{\substack{m \in \N^N \\ \sigma \in \Sigma(m,w,w')}}\int_{S_{m,\sigma}^\Delta}f(u)\psi_\lambda(u)\,du.
    \]
    
    
    Let $u \in S_{m,\sigma}^{\Delta}$. We use coordinates $y_1,y_2,\dots,y_N$ from \Cref{lem:iwahoridecompositionstep}. We introduce new coordinates $u_1,u_2,\dots,u_N$ as follows: if $y_k \not\in \OO_F$ (equivalently, $m_k > 0$), we write $y_k = \varpi^{-m_k}u_k$ for $u_k \in \OO_F^\times$, and if $y_k \in \OO_F$ (equivalently, $m_k = 0$), then $y_k = u_k$ for $u_k \in \OO_F$. We also need coordinates $w_1,w_2,\dots,w_N$ defined by $w_k = u_k$ if $m_k > 0$ and $w_k = 1$ if $m_k = 0$.

    We first give the explicit expression for function $f$ and the character $\psi_\lambda$ in terms of coordinates $u_1, u_2, \dots, u_N$ and $w_1,w_2,\dots,w_N$.
    
    \begin{lemma}\label{lem:explicitf}
        Let $u \in S_{m,\sigma}^\textbf{i}$ for arbitrary long word decomposition $\textbf{i}$. Then
        \[
            f(u)\,du = \left(\prod_{\alpha \in \Phi^+}z^{m_\alpha \alpha}(u_\alpha, \varpi)_{ij}^{m_\alpha + \sum_{\alpha <_\textbf{i} \beta}\langle \beta, \alpha^\vee \rangle m_\beta}\right)du_1du_2\dots du_N,
        \]
        where
        \[
            (\cdot, \cdot)_{i,j} = \begin{cases}
    			q^{-1}(\cdot, \cdot), &\text{if $m_{i,j} > 0$}\\
    			1, &\text{otherwise}
    		\end{cases}.
        \]
    \end{lemma}
    \begin{proof}
        By \Cref{lem:mcnamaracellrelation}, we have $S_{m,\sigma}^{\textbf{i}} \subset C_m^{\textbf{i}}$, and then it is Lemma 6.3 in \cite{McN11}. Note that typo in this lemma (there is no Hilbert symbol in the formula). The correct value of $f$ is given in the same paper in the proof of Theorem 8.4. The idea of the proof is prove it by induction using relation from \cref{eq:hrelations}:
        \[
            h_\alpha(x)h_\beta(y)h_\alpha(x)^{-1} = h_\beta(y)(x,y)^{ \langle \beta, \alpha^\vee \rangle l_\beta}.
        \]
        In our case $l_\beta = 1$ for all $\beta \in \Phi^\vee$ due to the choice of the metaplectic cover.
    \end{proof}
    
    \begin{lemma}\label{lem:explicitpsi}
        Let $u \in S_{m,\sigma}^\Delta$. Then
        \[
            \psi_\lambda(u) = \prod_{\alpha \in \Phi^+}\psi_{ij}\left(\varpi^{s_{ij}}u_{ij}\frac{\prod_{k=j+1}^{r+1} w_{i,k}}{\prod_{k=j}^{r} w_{i+1,k+1}}\right),
        \]
        where
        \[
            \psi_{i,j} = \begin{cases}
    			\psi, &\text{if $j=i+1$}\\
    			\psi, &\text{if $m_{i+1,j} > 0$}\\
    			\psi, &\text{if $m_{i+1,j} = 0$ and $\sigma_{i+1,j}^{-1}(i) < \sigma_{i+1,j}^{-1}(i+1)$}
    			\\
    			1, & \text{otherwise}
    		\end{cases}.
    	\]
    \end{lemma}
    \begin{proof}
        Analogous to Proposition 8.2 and the beginning of Theorem 8.4 in \cite{McN11}. The idea of the proof is to follow how the entries of $u$ change with each step of \Cref{lem:iwahoridecompositionstep} in coordinates $y_k,u_k,w_k$ and write the value of $\psi_\lambda = \prod_{i=1}^{r}\psi(\varpi^{\Lambda_i} x_{i,i+1})$ explicitly. 
    \end{proof}
    
    By lemmas above, we have 
    \[
        \int_{S_{m,\sigma}^\Delta}f(u)\psi_\lambda(u)\,du = \left(\prod_{\alpha \in \Phi^+}z^{m_\alpha \alpha}\right)I(m,\sigma),
    \]
    
    where by $I(m,\sigma)$ we denote
    \begin{equation}\label{eq:summandintegral}
        \int_{S_{m,\sigma}^{\Delta}}\prod_{\alpha \in \Phi^+}\psi_{ij}\left(\varpi^{s_{ij}}u_{ij}\frac{\prod_{k=j+1}^{r+1} w_{i,k}}{\prod_{k=j}^{r} w_{i+1,k+1}}\right)(u_\alpha,\varpi)_{ij}^{m_\alpha + \sum_{\alpha <_\textbf{i} \beta}\langle \beta, \alpha^\vee \rangle m_\beta}\,du_{12}du_{13}\dots du_{r,r+1}.
    \end{equation}
    
    The change of variables 
    \[
        t_{ij} = u_{ij}\frac{\prod_{k=j+1}^{r+1} w_{i,k}}{\prod_{k=j}^{r} w_{i+1,k+1}}.
    \]
    transforms the integral into
    \[
        I(m, \sigma) = \prod_{\alpha \in \Phi^+}I_\alpha(m,\sigma),
    \]
    where 
    \[
        I_\alpha(m,\sigma) = \int_{D_{ij}}(t_\alpha, \varpi)^{r_{\alpha}}\psi(\varpi^{s_\alpha}t_\alpha)\,dt_\alpha,
    \]
    and
    \[
        s_{i,j} = \Lambda_i + \sum_{k=j}^{r}m_{i+1,k+1} - \sum_{k=j}^{r+1}m_{i,k}, \quad 
        r_{i,j} = \sum_{k \leq i}m_{k,j},
    \]
    and the domains $D_{i,j} = D_{i,j}(m,\sigma)$ are defined by 
    \[
        D_{i,j} \coloneqq \begin{cases}
        	\OO^\times, &\text{if $m_{i,j} > 0$}\\
        	\OO, &\text{if $m_{i,j} = 0$, $\sigma_{i+1,j}^{-1}(i) < \sigma_{i+1,j}^{-1}(i+1)$}\\
        	\OO^\times, &\text{if $m_{i,j} = 0$, $\sigma_{i+1,j}^{-1}(i) > \sigma_{i+1,j}^{-1}(i+1)$, and $\sigma_{i,j} = \sigma_{i+1,j}$}\\
        	\p, &\text{if $m_{i,j} = 0$, $\sigma_{i+1,j}^{-1}(i) > \sigma_{i+1,j}^{-1}(i+1)$, and $\sigma_{i+1,j} = s_{i}\sigma_{i+1,j}$}
        \end{cases}.
    \]
    
    See Section 8 of \cite{McN11} for details of this computation.
    
    Recall from \Cref{def:lusztigdata} that $\Lu(\lambda+\rho)$ denotes the set of Lusztig data corresponding to $\lambda+\rho$. It is a finite set for any weight $\lambda$. 
	\begin{lemma}
	    The integral $I(m,\sigma) = 0$ unless $m \in \Lu(\lambda + \rho)$. 
	\end{lemma}
	\begin{proof}
	    Let $m \in \N^{N}$ be not in $\Lu(\lambda + \rho)$. Let $(i,j) \in \Phi^+$ be such that $s_{i,j} < -1$. Recall that character $\psi$ is trivial on $\OO$ and nontrivial on $\varpi^{-1}\OO$.
	
	   \begin{description}
			\item[Case 1] Assume that $\psi_{i,j} = \psi$ in \Cref{lem:explicitpsi}. Then $I_\alpha(m,\sigma)$ becomes
			\[
			\int_{D_{i,j}}(t_{i,j},\varpi)^{r_{i,j}}_{i,j}\psi(\varpi^{s_{i,j}}t_{i,j})\,dt_{i,j}.
			\]
			Then no matter if $D_{i,j}$ is $\OO$, $\OO^\times$, or $\p$, the condition $s_{i,j} < -1$ implies that $I_\alpha(m,\sigma) = 0$ as an integral of a nontrivial character over a compact subgroup. 
			
			\item[Case 2] Assume that $\psi_{i,j} = 1$ in \Cref{lem:explicitpsi}. Note that $m_{i+1, j} = 0$ implies that $s_{i,j} = s_{i,j-1} + m_{i,j-1}$, and since $m_{i,j-1} \geq 0$, we have $s_{i,j-1} < -1$. Then consider $I_{i,j-1}(m,\sigma)$ instead. By above, $I_{i,j-1}(m,\sigma)$ is zero unless $m_{i+1,j-1} = 0$. Without loss of generality we assume that $m_{i+1,j-1} \neq 0$ (if it exists) by decreasing $j$ if necessary. In other words, we get that $I_{i,j-1}(m,\sigma) = 0$. 
		\end{description}
    	In any case, $I(m,\sigma) = \prod_{\alpha \in \Phi^+}I_\alpha(m,\sigma) = 0$ if $m$ is not in $\Lu(\lambda+\rho)$.
	\end{proof}
	
	\begin{remark}\label{rem:mcnamaraconjecture}
	    McNamara conjectures that (Remark 8.5 in \cite{McN11}) that the analogous sum for the spherical Whittaker function is always finite for arbitrary root system and arbitrary decomposition $\textbf{i}$. Leslie in \cite{Les19} shows that it is not always the case and gives an example of a long word decomposition for $G_2$ that has infinitely many non-zero terms in the corresponding evaluation. 
	\end{remark}
	
	Next, we compute integrals $I_\alpha(m, \sigma)$ explicitly. The integrals will be expressed in terms of the \textit{(normalized) Gauss sum} corresponding to character $\psi$ which is given by
    \[
        g(a, b)	= q^{-1}\int_{\mathcal{O}_F^\times}(t, \varpi)^a\psi(\varpi^b t)\,dt, \quad a, b \in \Z.
    \]
    
    Standard manipulations with the integral give the following explicit values:
    \[
    	g(a, b) = \begin{cases} 
    		0, \quad &\text{if $b < -1$},\\
    		-q^{-1}, \quad &\text{if $b=-1$ and $a \equiv 0 \pmod{n}$},\\
    		1-q^{-1}, \quad &\text{if $b \geq 0$ and $a \equiv 0 \pmod{n}$},\\
    		0, \quad &\text{if $b \geq 0$ and $a \not\equiv 0 \pmod{n}$}
    	\end{cases}.
    \]
    
    Moreover, if $a \not\equiv 0 \pmod{n}$, then $|g(a, -1)| = q^{1/2}$. 
    
    \begin{lemma}\label{lem:integralcalculation}
        Let $m \in \Lu(\lambda + \rho)$. Then $I_\alpha(m,\sigma) = \G(m,\sigma,\alpha)$, where
        
        \begin{equation}\label{eq:contribution}
            \G(m,\sigma,\alpha) = \begin{cases}
                \begin{cases}
    				g(r_{i,j}, s_{i,j}), &\text{if $m_{i,j} > 0$}\\
    				1, &\text{if $m_{i,j} = 0$ and  $D_{i,j} = \OO$ and $s_{i,j} \geq 0$}\\
    				0, &\text{if $m_{i,j} = 0$ and $D_{i,j} = \OO$ and $s_{i,j} = -1$}\\
    				1-q^{-1}, &\text{if $m_{i,j} = 0$ and $D_{i,j} = \OO^\times$ and $s_{i,j} \geq 0$}\\
    				-q^{-1}, &\text{if $m_{i,j} = 0$ and $D_{i,j} = \OO^\times$ and $s_{i,j} = -1$}\\
    				q^{-1}, &\text{if $m_{i,j} = 0$ and $D_{i,j} = \p$ and $s_{i,j} \geq 0$}\\
    				q^{-1}, &\text{if $m_{i,j} = 0$ and $D_{i,j} = \p$ and $s_{i,j} = -1$}
        		\end{cases}, \quad &\text{if $\psi_{i,j} = \psi$}\\
                \begin{cases}
    				g(r_{i,j}, 0), &\text{if $m_{i,j} > 0$}\\
    				1, &\text{if $m_{i,j} = 0$ and $D_{i,j} = \OO$}\\
    				1-q^{-1}, &\text{if $m_{i,j} = 0$ and $D_{i,j} = \OO^\times$}\\
    				q^{-1}, &\text{if $m_{i,j} = 0$ and $D_{i,j} = \p$}
    			\end{cases}, \quad &\text{if $\psi_{i,j} = 1$}
            \end{cases}.
        \end{equation}
    \end{lemma}
    \begin{proof}
        We mindlessly compute it case by case. Note that we have $s_{i,j} \geq -1$ for all $(i,j) \in \Phi^+$.
    	
    	\begin{description}
    		\item[Case 1] Assume that $\psi_{i,j} = \psi$ in \Cref{lem:explicitpsi}.
    		
    		\begin{description}
    			\item[Case 1.1] Let $m_{i,j} > 0$. Then the integral $I_{i,j}(m,\sigma)$ becomes 
    			\[
    			q^{-1}\int_{\OO^\times}(t_{i,j}, \varpi)^{r_{i,j}}\psi(\varpi^{s_{i,j}}t_{i,j}) \, dt_{i,j} = g(r_{i,j}, s_{i,j}),
    			\]
    			by the definition of the normalized Gauss sum. 
    			\item[Case 1.2] Let $m_{i,j} = 0$. Then the integral $I_{i,j}(m,\sigma)$ becomes
    			\[
    			\int_{D_{i,j}}\psi(\varpi^{s_{i,j}}t_{i,j})\,dt_{i,j} = \begin{cases}
    				1, &\text{if $D_{i,j} = \OO$ and $s_{i,j} \geq 0$}\\
    				0, &\text{if $D_{i,j} = \OO$ and $s_{i,j} = -1$}\\
    				1-q^{-1}, &\text{if $D_{i,j} = \OO^\times$ and $s_{i,j} \geq 0$}\\
    				-q^{-1}, &\text{if $D_{i,j} = \OO^\times$ and $s_{i,j} = -1$}\\
    				q^{-1}, &\text{if $D_{i,j} = \p$ and $s_{i,j} \geq 0$}\\
    				q^{-1}, &\text{if $D_{i,j} = \p$ and $s_{i,j} = -1$}
    			\end{cases},
    			\]
    		\end{description}
    		
    		\item[Case 2] Assume that $\psi_{i,j} = 1$ in \Cref{lem:explicitpsi}.
    		
    		\begin{description}
    			\item[Case 2.1] Let $m_{i,j} > 0$. Then the integral $I_\alpha(m,\sigma)$ becomes
    			\[
    			q^{-1}\int_{\OO^\times}(t_{i,j}, \varpi)^{r_{i,j}} \, dt_{i,j} = g(r_{i,j}, 0).
    			\]
    			
    			\item[Case 2.2] Let $m_{i,j} = 0$. Then the integral $I_\alpha(m,\sigma)$ becomes 
    			\[
    			\int_{D_{i,j}}\, dt_{i,j} = \begin{cases}
    				1, &\text{if $D_{i,j} = \OO$}\\
    				1-q^{-1}, &\text{if $D_{i,j} = \OO^\times$}\\
    				q^{-1}, &\text{if $D_{i,j} = \p$}
    			\end{cases}.
    			\]
    		\end{description}
    	\end{description}
    	Thus, $I_\alpha(m,\sigma) = G(m,\sigma,\alpha)$, and it finishes the proof.
    \end{proof}
    
    We summarize the proof above.
    
    \begin{theorem}[Evaluation of Iwahori Whittaker functions]\label{thm:miw}
        Let $\lambda \in \Lambda$ with $\lambda_{r+1} = 0$ and let $w, w' \in W$. Then the integrals $\phi_w(\lambda, w'; z)$ defined by \cref{eq:miw} that determines values of Iwahori Whittaker functions is given by 
        \begin{equation*}
            \phi_w(\lambda, w'; z) = \sum_{\textbf{m} \in \LLu(\lambda+\rho,w,w')}\prod_{\alpha \in \Phi^+}\G(\textbf{m},\alpha; q)z^{m_\alpha \alpha},
        \end{equation*}
        where contributions $\G(\textbf{m},\alpha; q) \in \Z[q^{-1}]$ are given explicitly by \eqref{eq:lusztigcontribution}.
    \end{theorem}
    \begin{proof}
        We have
        \begin{align*}
            \phi_w(\lambda, w'; z) 
            &= \int_{U} f_w(uw')\psi_\lambda(u)\,du \\
            &= \sum_{m,\sigma} \int_{S_{m,\sigma}^{\Delta}}f(u)\psi_\lambda(u)\,du \\
            &= \sum_{m,\sigma}\left(\prod_{\alpha \in \Phi^+}z^{m_\alpha \alpha}\right)I(m,\sigma) \\
            &= \sum_{m,\sigma}\prod_{\alpha \in \Phi^+}z^{m_\alpha \alpha}\prod_{\alpha \in \Phi^+}I_\alpha(m,\sigma) \\
            &= \sum_{m,\sigma}\prod_{\alpha \in \Phi^+}\G(m,\sigma,\alpha)z^{m_\alpha \alpha}.
        \end{align*}
        By \Cref{section:coloreddata}, the pairs $(m,\sigma)$ are parametrized by colored Lusztig data $\LLu(\lambda+\rho,w,w')$ and the contributions $\G(m,\sigma,\alpha)$ from \cref{eq:contribution} can be written by $\G(\textbf{m},\alpha)$ from \cref{eq:lusztigcontribution}. It concludes the proof.
    \end{proof}
    
    \subsection{Other metaplectic covers}
    At the beginning of the section we chose a specific metaplectic cover that corresponds to the bilinear form $B$ defined by $B(e_i,e_j) = \delta_{ij}$. Now we explore the situation of arbitrary metaplectic cover. See \cite{Fre20} for details.
    
    Let $\widetilde{G}$ be a metaplectic cover of $G$ from \Cref{section:metaplecticgroups} corresponding to the bilinear linear form $B$. The main difference is that $l_{\alpha} = Q(\alpha^\vee) = B(\alpha^\vee, \alpha^\vee)/2$ enters the formula for multiplication on the torus $\widetilde{T}$ in \cref{eq:hrelations}:
    \[
        h_\alpha(x)h_\beta(y)h_\alpha(x)^{-1} = h_\beta(y)(x, y)^{B(\beta,\alpha^\vee)}. 
    \]
    
    Then in \Cref{lem:explicitf}, we have the explicit expression for function $f$ as follows:
    \[
        f(u)\,du = \left(\prod_{\alpha \in \Phi^+}z^{m_\alpha \alpha}(u_\alpha, \varpi)_{ij}^{m_\alpha + \sum_{\alpha <_\Delta \beta}B(\beta,\alpha^\vee) m_\beta} \right).
    \]
    Note that the only difference is that $\langle \beta, \alpha^\vee \rangle$ is replaced with $B(\beta,\alpha^\vee)$ which comes from the multiplication on the torus. It changes the formula for $I_\alpha(m,\sigma)$:
    \[
        I_\alpha(m,\sigma) = \int_{D_{ij}}(t_\alpha, \varpi)^{Q(\alpha^\vee) r_\alpha}\psi(\varpi^{s_\alpha}t_\alpha)\,dt_\alpha.
    \]
    
    It results the change in the formula for weights. Each instance of the normalized Gauss sum $g(r_{\alpha},s_{\alpha})$ will be change to $g(Q(\alpha^\vee)r_{\alpha}, s_{\alpha})$. This is the only change. 
    
    \section{Examples}\label{section:examples}
    In this section, we use \Cref{thm:miw} to compute values $\phi_w(\lambda,w'; z)$ given by \cref{eq:miw} which determine the metaplectic Iwahori Whittaker functions for the general linear group. For brevity, in all examples we write $t = -q^{-1}$. We write examples in terms of colored GT-patterns, but by \Cref{thm:bijection}, they are equivalent to colored Lusztig data or colored lattice models. 
    
    Let $\mu$ be a weight. We denote $a_{i,r+2} = \mu_{i}$ for each $i \in 1,\dots,r+1$. A GT-pattern with the top row $\mu$ looks like this: 
    \[
    	\left\{\begin{matrix}
    
        \mu_1 &  & \mu_2 & & \dots & & \mu_{r} & & \mu_{r+1}\\
    	& a_{1,r+1} & & a_{2,r+1} & \dots & a_{r-1,r+1} & & a_{r,r+1} & \\
    	&  & a_{1,r} & & \dots & & a_{r-1,r} & & \\
    	&  &  \ddots&  & &  &\udots &\\
    	&  &  & a_{1,3} & & a_{2,3} & & \\
    	&  & & & a_{1,2} & & & 
        
        \end{matrix}\right\}
    \]

    Recall that the statistics $s_{ij}$ are defined by
    \[
        s_{ij} = a_{ij} - a_{i+1,j+1} - 1.
    \]
    In particular, $s_{ij} = -1$ if an entry $a_{ij}$ is right-leaning; otherwise $s_{ij} \geq 0$. We also recall that statistics $r_{ij}$ are defined by
    \[
        r_{i,j} = \sum_{k \leq i}(a_{i,j+1}-a_{ij}).
    \]
    
    To have an additional check to our computations, we use
	\begin{theorem}[\cite{CS80}]
	    The value of the spherical Whittaker function on the diagonal element $\varpi^{\lambda}$ equals to
	    \begin{equation}\label{eq:csformula}
	        W_0(\varpi^\lambda) = \sum_{w \in W}W_w(\varpi^\lambda w') = \prod_{\alpha \in \Phi^+}(1+t z_\alpha)s_\lambda(z), \quad \text{for any $w' \in W$}.
	    \end{equation}
	\end{theorem}
	
	Therefore, in non-metaplectic case $n=1$, the sum over the Weyl group of Iwahori Whittaker functions will give us the closed Casselman-Shalika expression.
	
	Let $G = \GL_{2}$, so $r = 1$. We use the template
    $\left\{\begin{smallmatrix}
        a_{13} &        & a_{23} \\
               & a_{12} &      
    \end{smallmatrix}\right\}$ to read off statistics $s_{ij}$ and $r_{ij}$: we get $s_{12} = a_{12}-a_{23}-1$ and $r_{12} = a_{13}-a_{12}$. The Weyl group consists just of two elements: identity and $s_1$. The Weyl vector is $\rho = (1,0)$. 
    
    \begin{example}
        Let $G = \GL_{2}$, so $r = 1$. Let $\lambda = (0,0)$. Then $\lambda + \rho = (1,0)$. Here are all colored GT-patterns:
        
        \begin{description}
            \item[Input $w'=1$, output $w=1$]  
            $\left\{\begin{smallmatrix}
                {}_1 1 &        & {}_2 0 \\
                       & {}_2 0 &      
            \end{smallmatrix}\right\}$;
            \item[Input $w'=1$, output $w=s_1$]
            $\left\{\begin{smallmatrix}
                {}_1 1 &        & {}_2 0 \\
                       & {}_1 1 &      
            \end{smallmatrix}\right\}$;
            \item[Input $w'=s_1$, output $w=1$]
            $\left\{\begin{smallmatrix}
            {}_2 1 &        & {}_1 0 \\
                   & {}_2 1 &      
            \end{smallmatrix}\right\}$;
            \item[Input $w'=s_1$, output $w=s_1$]
            $\left\{\begin{smallmatrix}
            {}_2 1 &        & {}_1 0 \\
                   & {}_1 1 &      
            \end{smallmatrix}\right\}$ and $\left\{\begin{smallmatrix}
            {}_2 1 &        & {}_1 0 \\
                   & {}_1 0 &      
            \end{smallmatrix}\right\}$.
        \end{description}
        
        \Cref{thm:miw} and \cref{eq:contribution} provide the values of all Iwahori Whittaker functions:
        \[
            \begin{array}{ll}
                \phi_1(\lambda,1) = g(1,-1)z_1, & \phi_{s_1}(\lambda,1) = z_2,  \\
                \phi_1(\lambda,s_1) = (-t)z_2, & \phi_{s_1}(\lambda,s_1) = (1+t)z_2 + g(1,-1)z_1.
            \end{array}
        \]
        
        In a non-metaplectic case, $g(r_{12},-1) = t$. Then the calculation is consistent with \cref{eq:csformula}:
        \begin{align*}
            \phi_1(\lambda,1) + \phi_{s_1}(\lambda,1) = tz_1 + z_2\\
            \phi_1(\lambda,s_1) + \phi_{s_1}(\lambda,s_1) = (-t)z_2 + (1+t)z_2 + tz_1 = tz_1 + z_2.
        \end{align*}
    \end{example}
    
    More generally, we have
    \begin{example}
        Let $G = \GL_{2}$, so $r = 1$. Let $\lambda = (a,0)$ for $a \in \N$. Then $\lambda + \rho = (a+1,0)$. Here are all colored GT-patterns:
        
        \begin{description}
        \item[Input $w'=1$, output $w=1$]  
        $\left\{\begin{smallmatrix}
            {}_1 (a+1) &        & {}_2 0 \\
                   & {}_2 0 &      
        \end{smallmatrix}\right\}$, $\left\{\begin{smallmatrix}
            {}_1 (a+1) &        & {}_2 0 \\
                   & {}_2 1 &      
        \end{smallmatrix}\right\}$, \dots, $\left\{\begin{smallmatrix}
            {}_1 (a+1) &        & {}_2 0 \\
                   & {}_2 a &      
        \end{smallmatrix}\right\}$;
        \item[Input $w'=1$, output $w=s_1$]
        $\left\{\begin{smallmatrix}
            {}_1 (a+1) &        & {}_2 0 \\
                   & {}_1 (a+1) &      
        \end{smallmatrix}\right\}$;
        \item[Input $w'=s_1$, output $w=1$]
        $\left\{\begin{smallmatrix}
        {}_2 (a+1) &        & {}_1 0 \\
               & {}_2 (a+1) &      
        \end{smallmatrix}\right\}$;
        \item[Input $w'=s_1$, output $w=s_1$]
        $\left\{\begin{smallmatrix}
        {}_2 (a+1) &        & {}_1 0 \\
               & {}_1 0 &      
        \end{smallmatrix}\right\}$, $\left\{\begin{smallmatrix}
        {}_2 (a+1) &        & {}_1 0 \\
               & {}_1 1 &      
        \end{smallmatrix}\right\}$, \dots, $\left\{\begin{smallmatrix}
        {}_2 (a+1) &        & {}_1 0 \\
               & {}_1 a &      
        \end{smallmatrix}\right\}$, and \\ $\left\{\begin{smallmatrix}
        {}_2 (a+1) &        & {}_1 0 \\
               & {}_1 (a+1) &      
        \end{smallmatrix}\right\}$.
        \end{description}
        
        \Cref{thm:miw} and \cref{eq:contribution} provide values of all Iwahori Whittaker functions:
        \begin{align*}
            \phi_1(\lambda,1) &= g(a+1,-1)z_1^{a+1} + g(a,0)z_1^{a}z_2 + \dots + g(1,a-1)z_1z_2^a,\\
            \phi_{s_1}(\lambda,1) &= z_2^{a+1},\\
            \phi_1(\lambda,s_1) &= (-t)z_2^{a+1},\\
            \phi_{s_1}(\lambda,s_1) &=  g(a+1,-1)z_1^{a+1} + g(a,0)z_1^az_2 + \dots + g(1,a-1)z_1z_2^{a} + (1+t)z_2^{a+1}.
        \end{align*}
        In a non-metaplectic case, $g(r_{12},-1) = t$ and $g(r_{12},0) = (1+t)$ for any $r_{12}$. Then the calculation is consistent with \cref{eq:csformula}:
        \begin{align*}
            \phi_1(\lambda,1) + \phi_{s_1}(\lambda,1) &= tz_1^{a+1} + (1+t)z_1^az_2 + \dots + (1+t)z_1z_2^a + z_2^{a+1} \\
            &= (tz_1 + z_2)s_\lambda(z)\\
            \phi_1(\lambda,s_1) + \phi_{s_1}(\lambda,s_1) &= (-t)z_2^{a+1} + tz_1^{a+1} + (1+t)z_1^az_2 + \dots + (1+t)z_1z_2^a + (1+t)z_2^{a+1} \\
            &= (tz_1 + z_2)s_\lambda(z),
        \end{align*}
        where $s_\lambda(z) = z_1^{a}+z_1^{a-1}z_2 + \dots + z_2^a$. 
    \end{example}
    
    The next example shows that $\lambda$ should not necessary be a dominant weight to give a non-zero Iwahori Whittaker function. But by \Cref{lem:almostcondition}, $W_w(g) = 0$ unless $\lambda$ is $w'$-almost dominant. 
    
    \begin{example}
        Let $G = \GL_{2}$, so $r=1$. Let $\lambda = (0, 1)$. Then $\lambda + \rho = (1,1)$. There are only two colored GT-patterns: 
        
        \begin{description}
        \item[Input $w'=s_1$, output $w=1$]
        $\left\{\begin{smallmatrix}
        {}_2 1 &        & {}_1 1 \\
               & {}_2 1 &      
        \end{smallmatrix}\right\}$.
        \item[Input $w'=s_1$, output $w=s_1$]  
        $\left\{\begin{smallmatrix}
        {}_2 1 &        & {}_1 1 \\
               & {}_1 1 &      
        \end{smallmatrix}\right\}$.
        
        \end{description}
        
        We get the values of Iwahori Whittaker functions:
        \begin{equation*}
            \phi_1(\lambda,s_1) = -tz_1, \quad \phi_{s_1}(\lambda,s_1) = tz_1.
        \end{equation*}
        
        The sum of Iwahori Whittaker function is zero which is consistent with \cref{eq:csformula} as the spherical Whittaker function is zero for non-dominant weights.
    \end{example}
    
    \begin{remark}\label{rem:cancellations}
    	Note the non-trivial cancellations happen when we have choice for coloring an entry. These cancellations are ``invisible'' when calculating the spherical Whittaker function directly. In particular, a non-strict GT-pattern that has zero contribution to the spherical Whittaker function splits to two terms, each giving a non-zero contribution to the corresponding Iwahori component. Recall that the coloring choice comes from the splitting of $\OO$ into $\OO^\times$ and $\p$ when computing Iwahori Whittaker functions.
    \end{remark}
    
    \begin{example}
        Let $\lambda$ be a weight, $w,w'$ be permutations such that $\lambda$ is $w'$-almost dominant and $w = w_0w'$. Then $\phi_w(\lambda,w') = (q^{-1})^{l(w)} z^{\lambda+\rho}$ as there is only one admissible state in the colored data with output equals to the inverse of the input. It is Proposition 3.6 in \cite{BBBG19}.
    \end{example}
    
    For the next example, we write a table of all colored GT-patterns together with the corresponding colored Lusztig data to demonstrate the weight-preserving bijection.
    
    \begin{example}
        Let $G = \GL_{3}$, so $r = 2$. Let $\lambda = (1,0,0)$, and let the input be $w'=1$. Let us work in the non-metaplectic case $n = 1$. Then $g(r_{ij}, -1) = t$ and $g(r_{i,j},0) = (1+t)$ for any $r_{i,j}$. It is convenient to write all colored Lusztig data $\LLu(\lambda+\rho)$ with input $w'$ in a table together with weights and contributions. Then the value $\phi_w(\lambda,w'; z)$ is the corresponding sum over rows with output $w \in W$. See the table on the next page. This time \Cref{thm:miw} and \cref{eq:contribution} provide all diagonal values of the Iwahori Whittaker functions:
        
        \begin{align*}
            \phi_1(\lambda, 1) &= (-t)(1+t)z_1^2z_2z_3 + (-t)(1+t)z_1z_2z_3^2 + t(1+t)^2z_1z_2z_3^2 + t^3z_2^3z_3 + t^2(1+t)z_2^2z_3^2, \\
            \phi_{s_1}(\lambda, 1) &= t^2z_2^3z_3 + t^2z_1z_2^2z_3 + t(1+t)z_1^2z_2z_3,\\
            \phi_{s_2}(\lambda, 1) &= t^2z_1z_3^3 + (-t)(1+t)z_1z_2z_3^2 + t(1+t)z_1^2z_3^2 + t(1+t)z_1z_2^2z_3 + (-1)(1+t)z_1^2z_2z_3, \\
            \phi_{s_1s_2}(\lambda, 1) &= tz_1z_2^3 + (1+t)z_1^2z_2^2, \\
            \phi_{s_2s_1}(\lambda, 1) &= tz_1^3z_3, \\
            \phi_{s_2s_1s_2}(\lambda, 1) &= z_1^3z_2.
        \end{align*}
    	
    	Note that we again can match the Casselman-Shalika formula:
    	\begin{align*}
    		\sum_{w \in W}\phi_w(\lambda, 1) &= z^{-\rho}(1+tz_2/z_1)(1+tz_3/z_1)(1+tz_3/z_2)(z_1+z_2+z_3) \\
    		&= z^{-\rho}\prod_{\alpha \in \Phi^+}(1+t z_\alpha)s_\lambda(z).
    	\end{align*}
    \end{example}
    
    \afterpage{%
    \[\arraycolsep=6pt\def\arraystretch{1.8}
    \resizebox{\textwidth}{!}{$\displaystyle
        \begin{array}{|c|c|c|c|c|}
    
    \hline
    
    \GGT(\lambda+\rho) & \LLu(\lambda+\rho) & \text{output} & z^{\lambda+\rho}z^{m} & \text{weights }G(\textbf{m},\alpha)\\
    
    \hline
    	
    \left\{\begin{smallmatrix}
    	{}_13 &   & {}_21 &   & {}_30\\
    	& {}_2 1 &   & {}_3 0 &  \\
    	&   & {}_3 0 &   &
    \end{smallmatrix}\right\} &
    \left\{\begin{smallmatrix}
    	 - & -  & -\\
    	   & {}_2 2  & {}_3 1\\
    	   &    & {}_3 1
    \end{smallmatrix}\right\} &
    \begin{psmallmatrix}
        1 & 2 & 3 \\
        1 & 2 & 3
    \end{psmallmatrix} = 1 & 
    z_2z_3^3 &
    \left\{\begin{smallmatrix}
        t & t \\
          & t
    \end{smallmatrix}\right\}
    
    \\
    
    \left\{\begin{smallmatrix}
    	{}_13 &   & {}_21 &   & {}_30\\
    	& {}_2 1 &   & {}_3 0 &  \\
    	&   & {}_2 1 &   &
    \end{smallmatrix}\right\} &
    \left\{\begin{smallmatrix}
    	- & -  & -\\
    	& {}_2 2  & {}_3 1\\
    	&    & {}_2 0
    \end{smallmatrix}\right\} &
    \begin{psmallmatrix}
        1 & 2 & 3 \\
        1 & 3 & 2
    \end{psmallmatrix} = s_2 & 
    z_1z_3^3 &
    \left\{\begin{smallmatrix}
        t & t \\
          & 1
    \end{smallmatrix}\right\}

    \\
    
    \left\{\begin{smallmatrix}
    	{}_13 &   & {}_21 &   & {}_30\\
    	& {}_3 1 &   & {}_2 1 &  \\
    	&   & {}_2 1 &   &
    \end{smallmatrix}\right\} &
    \left\{\begin{smallmatrix}
    	- & -  & -\\
    	& {}_3 2  & {}_2 0\\
    	&    & {}_2 0
    \end{smallmatrix}\right\} &
    \begin{psmallmatrix}
        1 & 2 & 3 \\
        1 & 3 & 2
    \end{psmallmatrix} = s_2 & 
    z_1z_2z_3^2 &
    \left\{\begin{smallmatrix}
        1+t & 1 \\
          & -t
    \end{smallmatrix}\right\}

    \\
    
    \left\{\begin{smallmatrix}
    	{}_13 &   & {}_21 &   & {}_30\\
    	& {}_3 1 &   & {}_2 1 &  \\
    	&   & {}_3 1 &   &
    \end{smallmatrix}\right\} &
    \left\{\begin{smallmatrix}
    	- & -  & -\\
    	& {}_3 2  & {}_2 0\\
    	&    & {}_3 0
    \end{smallmatrix}\right\} &
    \begin{psmallmatrix}
        1 & 2 & 3 \\
        1 & 2 & 3
    \end{psmallmatrix} = 1 & 
    z_1z_2z_3^2 &
    \left\{\begin{smallmatrix}
        1+t & 1 \\
          & t
    \end{smallmatrix}\right\}

    \\
    
    \left\{\begin{smallmatrix}
    	{}_13 &   & {}_21 &   & {}_30\\
    	& {}_2 2 &   & {}_3 0 &  \\
    	&   & {}_3 0 &   &
    \end{smallmatrix}\right\} &
    \left\{\begin{smallmatrix}
    	- & -  & -\\
    	& {}_2 1  & {}_3 1\\
    	&    & {}_3 2
    \end{smallmatrix}\right\} &
    \begin{psmallmatrix}
        1 & 2 & 3 \\
        1 & 2 & 3
    \end{psmallmatrix} = 1 & 
    z_2^2z_3^2 &
    \left\{\begin{smallmatrix}
        1+t & t \\
          & t
    \end{smallmatrix}\right\}

    \\
    
    \left\{\begin{smallmatrix}
    	{}_13 &   & {}_21 &   & {}_30\\
    	& {}_2 2 &   & {}_3 0 &  \\
    	&   & {}_3 1 &   &
    \end{smallmatrix}\right\} &
    \left\{\begin{smallmatrix}
    	- & -  & -\\
    	& {}_2 1  & {}_3 1\\
    	&    & {}_3 1
    \end{smallmatrix}\right\} &
    \begin{psmallmatrix}
        1 & 2 & 3 \\
        1 & 2 & 3
    \end{psmallmatrix} = 1 & 
    z_1z_2z_3^2 &
    \left\{\begin{smallmatrix}
        1+t & t \\
          & 1+t
    \end{smallmatrix}\right\}

    \\
    
    \left\{\begin{smallmatrix}
    	{}_13 &   & {}_21 &   & {}_30\\
    	& {}_2 2 &   & {}_3 0 &  \\
    	&   & {}_2 2 &   &
    \end{smallmatrix}\right\} &
    \left\{\begin{smallmatrix}
    	- & -  & -\\
    	& {}_2 1  & {}_3 1\\
    	&    & {}_2 0
    \end{smallmatrix}\right\} &
    \begin{psmallmatrix}
        1 & 2 & 3 \\
        1 & 3 & 2
    \end{psmallmatrix} = s_2 & 
    z_1^2z_3^2 &
    \left\{\begin{smallmatrix}
        1+t & t \\
          & 1
    \end{smallmatrix}\right\}

    \\
    
    \left\{\begin{smallmatrix}
    	{}_13 &   & {}_21 &   & {}_30\\
    	& {}_3 2 &   & {}_2 1 &  \\
    	&   & {}_2 1 &   &
    \end{smallmatrix}\right\} &
    \left\{\begin{smallmatrix}
    	- & -  & -\\
    	& {}_3 1  & {}_2 0\\
    	&    & {}_2 1
    \end{smallmatrix}\right\} &
    \begin{psmallmatrix}
        1 & 2 & 3 \\
        1 & 3 & 2
    \end{psmallmatrix} = s_2 & 
    z_1z_2^2z_3 &
    \left\{\begin{smallmatrix}
        1+t & 1 \\
          & t
    \end{smallmatrix}\right\}

    \\
    
    \left\{\begin{smallmatrix}
    	{}_13 &   & {}_21 &   & {}_30\\
    	& {}_3 2 &   & {}_2 1 &  \\
    	&   & {}_2 2 &   &
    \end{smallmatrix}\right\} &
    \left\{\begin{smallmatrix}
    	- & -  & -\\
    	& {}_3 1  & {}_2 0\\
    	&    & {}_2 0
    \end{smallmatrix}\right\} &
    \begin{psmallmatrix}
        1 & 2 & 3 \\
        1 & 3 & 2
    \end{psmallmatrix} = s_2 & 
    z_1^2z_2z_3 &
    \left\{\begin{smallmatrix}
        1+t & 1 \\
          & -t
    \end{smallmatrix}\right\}

    \\
    
    \left\{\begin{smallmatrix}
    	{}_13 &   & {}_21 &   & {}_30\\
    	& {}_3 2 &   & {}_2 1 &  \\
    	&   & {}_3 2 &   &
    \end{smallmatrix}\right\} &
    \left\{\begin{smallmatrix}
    	- & -  & -\\
    	& {}_3 1  & {}_2 0\\
    	&    & {}_3 0
    \end{smallmatrix}\right\} &
    \begin{psmallmatrix}
        1 & 2 & 3 \\
        1 & 2 & 3
    \end{psmallmatrix} = 1 & 
    z_1^2z_2z_3 &
    \left\{\begin{smallmatrix}
        1+t & 1 \\
          & 1+t
    \end{smallmatrix}\right\}

    \\
    
    \left\{\begin{smallmatrix}
    	{}_13 &   & {}_21 &   & {}_30\\
    	& {}_1 3 &   & {}_3 0 &  \\
    	&   & {}_3 0 &   &
    \end{smallmatrix}\right\} &
    \left\{\begin{smallmatrix}
    	- & -  & -\\
    	& {}_1 0  & {}_3 1\\
    	&    & {}_3 3
    \end{smallmatrix}\right\} &
    \begin{psmallmatrix}
        1 & 2 & 3 \\
        2 & 1 & 3
    \end{psmallmatrix} = s_1 & 
    z_2^3z_3 &
    \left\{\begin{smallmatrix}
        1 & t \\
          & t
    \end{smallmatrix}\right\}

    \\
    
    \left\{\begin{smallmatrix}
    	{}_13 &   & {}_21 &   & {}_30\\
    	& {}_1 3 &   & {}_3 0 &  \\
    	&   & {}_3 1 &   &
    \end{smallmatrix}\right\} &
    \left\{\begin{smallmatrix}
    	- & -  & -\\
    	& {}_1 0  & {}_3 1\\
    	&    & {}_3 2
    \end{smallmatrix}\right\} &
    \begin{psmallmatrix}
        1 & 2 & 3 \\
        2 & 1 & 3
    \end{psmallmatrix} = s_1 & 
    z_1z_2^2z_3 &
    \left\{\begin{smallmatrix}
        1 & t \\
          & 1+t
    \end{smallmatrix}\right\}

    \\
    
    \left\{\begin{smallmatrix}
    	{}_13 &   & {}_21 &   & {}_30\\
    	& {}_1 3 &   & {}_3 0 &  \\
    	&   & {}_3 2 &   &
    \end{smallmatrix}\right\} &
    \left\{\begin{smallmatrix}
    	- & -  & -\\
    	& {}_1 0  & {}_3 1\\
    	&    & {}_3 1
    \end{smallmatrix}\right\} &
    \begin{psmallmatrix}
        1 & 2 & 3 \\
        2 & 1 & 3
    \end{psmallmatrix} = s_1 & 
    z_1^2z_2z_3 &
    \left\{\begin{smallmatrix}
        1 & t \\
          & 1+t
    \end{smallmatrix}\right\}

    \\
    
    \left\{\begin{smallmatrix}
    	{}_13 &   & {}_21 &   & {}_30\\
    	& {}_1 3 &   & {}_3 0 &  \\
    	&   & {}_1 3 &   &
    \end{smallmatrix}\right\} &
    \left\{\begin{smallmatrix}
    	- & -  & -\\
    	& {}_1 0  & {}_3 1\\
    	&    & {}_1 0
    \end{smallmatrix}\right\} &
    \begin{psmallmatrix}
        1 & 2 & 3 \\
        2 & 3 & 1
    \end{psmallmatrix} = s_2s_1 & 
    z_1^3z_3 &
    \left\{\begin{smallmatrix}
        1 & t \\
          & 1
    \end{smallmatrix}\right\}

    \\
    
    \left\{\begin{smallmatrix}
    	{}_13 &   & {}_21 &   & {}_30\\
    	& {}_1 3 &   & {}_2 1 &  \\
    	&   & {}_2 1 &   &
    \end{smallmatrix}\right\} &
    \left\{\begin{smallmatrix}
    	- & -  & -\\
    	& {}_1 0  & {}_2 0\\
    	&    & {}_3 2
    \end{smallmatrix}\right\} &
    \begin{psmallmatrix}
        1 & 2 & 3 \\
        3 & 1 & 2
    \end{psmallmatrix} = s_1s_2 & 
    z_1z_2^3 &
    \left\{\begin{smallmatrix}
        1 & 1 \\
          & t
    \end{smallmatrix}\right\}

    \\
    
    \left\{\begin{smallmatrix}
    	{}_13 &   & {}_21 &   & {}_30\\
    	& {}_1 3 &   & {}_2 1 &  \\
    	&   & {}_2 2 &   &
    \end{smallmatrix}\right\} &
    \left\{\begin{smallmatrix}
    	- & -  & -\\
    	& {}_1 0  & {}_2 0\\
    	&    & {}_2 1
    \end{smallmatrix}\right\} &
    \begin{psmallmatrix}
        1 & 2 & 3 \\
        3 & 1 & 2
    \end{psmallmatrix} = s_1s_2 & 
    z_1^2z_2^2 &
    \left\{\begin{smallmatrix}
        1 & 1 \\
          & 1+t
    \end{smallmatrix}\right\}

    \\
    
    \left\{\begin{smallmatrix}
    	{}_13 &   & {}_21 &   & {}_30\\
    	& {}_1 3 &   & {}_2 1 &  \\
    	&   & {}_1 3 &   &
    \end{smallmatrix}\right\} &
    \left\{\begin{smallmatrix}
    	- & -  & -\\
    	& {}_1 0  & {}_2 0\\
    	&    & {}_1 0
    \end{smallmatrix}\right\} &
    \begin{psmallmatrix}
        1 & 2 & 3 \\
        3 & 2 & 1
    \end{psmallmatrix} = s_2s_1s_2 & 
    z_1^3z_2 &
    \left\{\begin{smallmatrix}
        1 & 1 \\
          & 1
    \end{smallmatrix}\right\}\\
    
    \hline
    
    \end{array}$}
    \]
    \clearpage
    }
    
    \newpage

    \appendix
    \section{Metaplectic groups and Iwahori Whittaker Functions}\label{section:miwappendix}
	Our main object of study is the Iwahori Whittaker functions which are certain matrix coefficients of an unramified genuine principal series representations of metaplectic covers of a split reductive group over a non-archimedean field. We remark that in all discussions one can set the degree of the cover to be one and obtain results about the split reductive group itself. We give only the definitions we directly need, and for the details on representation theory side we refer to \cite{McN11}[Sections 2-5], \cite{McN12}, \cite{Moo68}, and \cite{Ste68}.

    \subsection{Metaplectic groups}\label{section:metaplecticgroups}
    The discussion of metaplectic groups is based on \cite{McN11}, \cite{McN12}, \cite{McN16}. As usual, we start with the way too familiar wall of text.
    
    Let $G$ be a split (connected) reductive group over a non-archimedian local field $F$ with ring of integers $\OO_F$. Let $\p$ be the maximal ideal of $\OO_F$ with uniformizer $\varpi \in \p$. Denote by $q$ the cardinality of $\OO_F/\p$, and the residue field by $\mathbb{F}_q = \OO_F/\p$. Let $T$ be a fixed split maximal torus of $G$, and let $\widehat{T}$ be the corresponding maximal torus of $\widehat{G}$, the Langlads dual group. Let $B$ be a Borel subgroup of $G$ containing $T$, $U$ be its unipotent radical, and let $K = G(\OO_F)$ be a maximal compact subgroup of $G$. Let $B^-, U$ be the corresponding opposite subgroups. 
    
    Let $\Phi$ be the root system of $G$ which we assume to be irreducible. Then $\Phi$ is a subset of character group $X^*(T)$. The dual root system $\Phi^\vee$ is a subset of the cocharacter group $X_*(T)$ which is identified with $X^*(\widehat{T})$. We use $\langle \cdot, \cdot \rangle \colon \Phi \times \Phi^\vee \to \Z$ to denote the canonical pairing between $\Phi$ and $\Phi^\vee$. If $\alpha \in \Phi$, the corresponding element of $\Phi^\vee$ is denoted $\alpha^\vee$. We will denote $\Lambda = X_*(T) = X^*(\widehat{T})$. Let $I$ the finite index set of simple roots. 
    
    \textit{A metaplectic $n$-fold cover of $G$} is a central extension of $G$ by $\mu_n$:
    \[
    1 \longrightarrow \mu_n \longrightarrow \widetilde{G} \xrightarrow{\mathmakebox[12pt]{p}} G(F) \longrightarrow 1.
    \]
    As a set, $\widetilde{G} = G \times \mu_n$ with the group multiplication given by a cocycle in $H^2(G, \mu_n)$. Here $p$ is the natural projection map $\widetilde{G} \to G(F)$. For any subgroup $H$ of $G$, we denote by $\widetilde{H}$ the induced covering group of $H$. E.g., $\widetilde{B} = p^{-1}(B)$ and $\widetilde{T} = p^{-1}(T)$. 
    
    If $x \in F$ and $\lambda \in \Lambda$, we will denote the image of $x$ in $T$ under $\lambda$ by $x^\lambda$. By abuse of notation, we will denote a representative in $\widetilde{T}$ by the same symbol $x^\lambda$.
    
    Let $n$ be a positive integer such that $q \equiv 1 \pmod{2n}$. It implies that $F^\times$ contains $2n$ distinct $2n$-th roots of unity. Let $\mu_n$ be the group of $n$-th roots of unity. Fix a faithful character $\varepsilon\colon \mu_n \to \C^\times$. Let $(\cdot, \cdot)\colon F^\times \times F^\times \to \mu_n$ be the $n$-th power Hilbert symbol. 
    
    Let $B \colon \Lambda \times \Lambda \to \Z$ be an $W$-invariant symmetric bilinear form on $\Lambda$ such that $Q(\alpha^\vee) \coloneqq B(\alpha^\vee, \alpha^\vee)/2 \in \Z$ for all coroots $\alpha^\vee$. (There will never be any possible confusion between this use of the symbol $B$ and its use for a Borel subgroup). 
    
    Under these assumptions on $n$ and bilinear form $B$, by Theorem 3.2 of \cite{McN12}, there exists an $n$-fold metaplectic cover $\widetilde{G}$ of $G(F)$ such that 
    \[
    	[x^\lambda, y^\mu] = (x, y)^{B(\lambda, \mu)}, \quad x,y \in F, \quad \lambda,\mu \in \Lambda,
    \]
    where $[\cdot, \cdot]$ is the group commutator and $(\cdot, \cdot) \in \mu_n$ is the $n$-th power Hilbert symbol. The identity does not depend on the choice of representatives for $x^\lambda$ and $y^\mu$. 
    
    Fix such a metaplectic cover $\widetilde{G}$ of $G(F)$.
    
    \subsection{Unramified genuine principal series}
    We recall the construction of the unramified genuine principal series representation of metaplectic group $\widetilde{G}$. We follow \cite{McN11, McN12, McN16}.
    
    Let $T(\OO_F) = \widetilde{T} \cap K$. Since the cover splits over $K$, the group $T(\OO_F)$ may be regarded as a subgroup of $G(F)$. Let $H = C_{\widetilde{T}}(T(\OO_F))$, the centralizer of $T(\OO_F)$ in $\widetilde{T}$. It is a maximal abelian subgroup of $\widetilde{T}$ by Lemma 1 of \cite{McN12}. 
    
    A function $f$ on $\widetilde{G}$ is called \textit{genuine} if $f(\zeta g) = \epsilon(\zeta)f(g)$ for $\zeta \in \mu_n$, $g \in \widetilde{G}$. A genuine quasicharacter of $H$ is called \textit{unramified} if it is trivial on $T(\OO_F)$. Let $\chi$ be an unramified quasicharacter of $H$, and set
    \[
    	i(\chi) = \Ind_H^{\widetilde{T}}(\chi) = \{f \colon \widetilde{T} \to \C \mid f(ht) = \chi(h)f(t) \text{ for all } t \in \widetilde{T}, h \in H\}.
    \]
    The group $\widetilde{T}$ acts on $i(\chi)$ by right translation. Denote this representation as $(\pi_\chi, i(\chi))$. It is an irreducible finite-dimensional $\widetilde{T}$-module by Theorem 5.1 of \cite{McN12}. Next, we inflate $i(\chi)$ from $\widetilde{T}$ to $\widetilde{B}$ and then induce (with normalization) to $\widetilde{G}$ to obtain 
    \[
    	I(\chi) = \Ind_{\widetilde{B}}^{\widetilde{G}}(i(\chi)) = \{\text{smooth } f\colon \widetilde{G} \to i(\chi) \mid f(bg) = (\delta^{1/2}\chi)(b)f(g) \text{ for all } b \in \widetilde{B}, g \in \widetilde{G}\},
    \]
    where $\delta$ is the modular quasicharacter of $\widetilde{B}$. The group $\widetilde{G}$ acts by right multiplication on $I(\chi)$. This representation is called \textit{unramified genuine principal series representation of $\widetilde{G}$} and denoted by $(\pi_\chi, I(\chi))$. We assume that $I(\chi)$ is irreducible.
    
    By Lemma 6.3 of \cite{McN12}, $I(\chi)$ has a one-dimensional space of $K$-fixed vectors. Fix a non-zero vector $\phi_K^\chi \in I(\chi)^K$ which is called a \textit{spherical vector}.
    
    \subsection{Whittaker functionals on metaplectic covers}
    We define Whittaker functionals on $\widetilde{G}$ following \cite{McN12}, \cite{McN16}.
    
    Choose Haar measure on $F$ such that $\OO$ has volume $1$, and denote it by $dx$. Choose a normalization of Haar measure on $U$ such that $du = \prod_{i=1}^{N}dx_i$.
    
    Fix a non-degenerate character $\psi\colon U \to \C$ of $U$ such that the restriction to the subgroup $U_{-\alpha} = \{e_{-\alpha}(x) \mid x \in F\} \cong F$ for each simple root $\alpha$ has conductor $\OO_F$, that is, trivial on $\OO_F^\times$, but non-trivial on $\varpi^{-1}\OO_F$. 
    
    \textit{A (complex-valued) Whittaker functional (with respect to $\psi$)} on a representation $(\pi, V)$ of $\widetilde{G}$ is a linear functional $W \colon V \to \C$ such that $W(\pi(u)v) = \psi(u)W(v)$ for all $u \in U$, $v \in V$.
    
    By \cite{McN16} (Section 6), there is a unique (up to constant) $i(\chi)$-valued function $W^\chi\colon I(\chi) \to i(\chi)$ on $(\pi_\chi, I(\chi))$ given by the integral
    \[
    	W^\chi(\phi) = \int_{U}\phi(u)\psi(u)^{-1}du.
    \]
    
    The integral is convergent if $|z|^{\alpha^\vee} < 1$ for all positive roots $\alpha$, and can be extended to all $z$ by analytic continuation. 
    
    To get a Whittaker functional we need to compose $W^\chi$ with a linear functional on $i(\chi)$. More precisely, there is an isomorphism between $i(\chi)^*$ and the space of $\C$-valued Whittaker functionals on $I(\chi)$ given by composition (Theorem 6.2 of \cite{McN16})
    \[
    	L \mapsto L \circ W^\chi, \quad L \in i(\chi)^*.
    \]
    Thus, the dimension of the space of Whittaker functionals on $I(\chi)$ is $\dim(i(\chi)) = |\widetilde{T}/H|$ by Theorem 8.1 of \cite{McN12}. 
    
    \begin{remark}
    	Note that in the non-metaplectic case ($n=1$), the space of Whittaker functionals is one-dimensional. In the metaplectic case ($n > 1$) we have a richer theory of Whittaker functionals.
    \end{remark}
    
    Now we construct a natural basis of the Whittaker functionals. Let $v_0 = \phi_K^\chi(1) \in i(\chi)$ and let $\Gamma$ be a subset of coweights such that $\{\varpi^\gamma\}_{\gamma \in \Gamma}$ is a set of coset representatives for $\widetilde{T}/H$. When unambiguous, we identify $\gamma \in \Gamma$ with this set. Then vectors $\{\pi_\chi(\varpi^\gamma)v_0\}_{\gamma \in \Gamma}$ form a basis of $i(\chi)$. Let $L^\chi_\gamma$ for $\gamma \in \Gamma$ denote the corresponding dual basis of $i(\chi)^*$, so we have 
    \[
    L_\gamma^\chi(\pi_\chi(\varpi^\lambda)v_0) = \begin{cases}
    	\chi(\varpi^{\lambda}), &\text{if $\varpi^{\lambda - \gamma} \in H$},\\
    	0, &\text{otherwise}
    \end{cases}.
    \]
    
    By above, the space of Whittaker functionals on $I(\chi)$ has the following basis
    \[
    	W^\chi_\gamma = L^\chi_\gamma \circ W^\chi, \quad \gamma \in \Gamma.
    \]
    
    \subsection{Iwahori Whittaker function for metaplectic covers}\label{section:iwahorifunctions}
    
    Define the Iwahori subgroup $J=J$ of $G(F)$ as the preimage of $B^-(\mathbb{F}_q)$ under the mod $\p$ reduction $G(\OO_F) \to G(\mathbb{F}_q)$. The space of Iwahori fixed vectors $I(\chi)^J$ has dimension $|W|$. For each $w \in W$, define an Iwahori fixed vector $\phi_w^\chi \in I(\chi)^J$ by
    \[
    \phi_w^\chi\left(\zeta b w'j\right) = 
    \begin{cases}
    	\varepsilon(\zeta)\phi_K^\chi(b), &\text{if $w' = w_0w$},\\
    	0, &\text{otherwise}
    \end{cases}.
    \]
    where $\phi_K$ is the spherical fixed above, and $\zeta \in \mu_n$, $b \in \widetilde{B}$, $w' \in W$, and $j \in J$; we use the metaplectic Iwahori decomposition $\widetilde{G} = \widetilde{B}WJ$ to express an arbitrary element as product of such terms. Note the twist by $w_0$ in the definition. By construction, $\phi_w^\chi$ form a basis of $I(\chi)^J$, and are sometimes called \textit{the standard Iwahori basis.}
    
    \begin{remark}
    	Note that $\phi_K^\chi = \sum_{w \in W}\phi_w^\chi$. Thus, we have a refinement of a spherical vector $\phi_K^\chi$, and we can obtain information about the spherical vector by summing over Iwahori basis.
    \end{remark}
    
    Now we can define the main object of our study, the (metaplectic) Iwahori Whittaker functions $\Omega_{w,\gamma}^\chi\colon \widetilde{G} \to \C$, which are the matrix coefficients on $I(\chi)$ corresponding to Whittaker basis functionals $W_\gamma^\chi$ and Iwahori basis vectors $\phi_w^\chi$. Explicitly, 
    \[
    	\Omega^\chi_{w,\gamma}(g) \coloneqq W^\chi_\gamma(\pi_\chi(g)\phi_w^\chi) = L^\chi_\gamma\left(\int_{U}\phi_w(u g)\psi(u)^{-1}\,du\right).
    \] 
    
    Note that $L_\gamma^\chi$ commutes with the integral. Write $g = utw'j$ with $u \in U^+$, $t \in \widetilde{T}$, $w' \in W$, and $j \in J$, and write $t = \varpi^\mu h$ for $\varpi^\mu \in \widetilde{T}/H$, $\mu \in \Gamma$, and $h \in H$. For convenience, let us write $L_\gamma^\chi \circ \phi_w^\chi$ as follows 
    \begin{align*}
    	L_\gamma^\chi(\phi_w^\chi(g))
    	&= L_\gamma^\chi(\phi_w^\chi(\varpi^\mu h)) \\
    	&= L_\gamma^\chi(\pi_\chi(\varpi^\mu)\phi_w^\chi(h)) \\
    	&= \begin{cases}
    	    L_\gamma^\chi(\delta^{1/2}(b)\chi(b)v_0), &\text{if $w = w_0w'$}\\
    	    0, &\text{otherwise}
    	\end{cases} \\
    	&= \begin{cases}
    	    \delta^{1/2}(b) L_\gamma^\chi(\chi(\varpi^{\mu})\chi(h) v_0), &\text{if $w = w_0w'$}\\
    	    0, &\text{otherwise}
    	\end{cases} \\
    	&= \begin{cases}
    		\delta^{1/2}(b)\chi(\varpi^\lambda), &\text{if $w = w_0w'$ and $\varpi^{\lambda-\gamma} \in H$}\\
    		0, &\text{otherwise}
    	\end{cases}.
    \end{align*}
    
    Denote $\chi(\varpi^\lambda) = z^\lambda$ for $z \in \C^{r+1}$ for each $\lambda \in \Lambda$. Then we write $f_{w,\gamma}^z = L_\gamma^\chi \circ \phi_w^\chi$. Note that $\chi$ does not depend uniquely on $z$: in particular, if $z^n = (z')^n$, then the corresponding representations are isomorphic. Furthermore, $I(\chi)$ is irreducible if and only if $z^{n\alpha} \neq q^{\pm 1}$ for all coroots $\alpha$. Finally, denote $f_w = f_w^z = \sum_{\gamma \in \Gamma}f_{w,\gamma}^z$, and $f = f^z = \sum_{w \in W}f_w^z$. Then in notaion above, $f_w(g) = f(g) = \delta^{1/2}(b)\chi(\varpi^\lambda)$ if $w = w_0w'$ and zero otherwise.
    
    The \textit{average metaplectic Iwahori Whittaker Function} $W_w = W_w^\chi \colon \widetilde{G} \to \C$ is given by
    \[
        W_w(g) = \int_{U}f_w(ug)\psi(u)\,du.
    \]
    
    \begin{definition}[Definition 2.1, \cite{BBBG20}]
        Let $w' \in W$, let $\alpha_i$ be a simple root and $\alpha_i^\vee$ the corresponding coroot. A weight $\lambda$ is $w'$-almost dominant if 
        \[
            \langle \alpha_i^\vee, \lambda \rangle \geq \begin{cases}
                0, \quad \text{if $(w')^{-1}\alpha_i \in \Phi^+$},\\
                -1,\quad \text{if $(w')^{-1}\alpha_i \in \Phi^-$}.
            \end{cases} \quad \text{for all simple roots $\alpha_i$}.
        \]
    \end{definition}
    
    \begin{lemma}[Lemma 3.7, \cite{BBBG20}]\label{lem:almostcondition}
        Let $\lambda \in \Lambda$ and $w' \in W$. Then $W_w(g) = 0$ unless $\lambda$ is $w'$-dominant. 
    \end{lemma}
    
    Let $g\in \widetilde{G}$. By Iwahori decomposition, we can write $g = n\varpi^{-\lambda} w'j$ with $n \in U$, $\lambda \in \Lambda$, $w' \in W$, and $j \in J$. Transform $W_w$ on the left by $\psi$ and on the right by $j$ trivially to up get 
    \[
        \int_{U} f_w(u\varpi^{\lambda}w')\psi(u)\,du
    \]
    up to a constant. We can assume $\lambda_{r+1} = 0$ as the diagonal element in central, and we just get a factor of $z^\lambda$. Conjugate $u \mapsto \varpi^{-\lambda}u\varpi^{\lambda}$ to get
    \[
        \int_{U} f_w(uw')\psi_\lambda(u)\,du = \phi_w(\lambda,w'; z),
    \]
    up to a constant, $\psi_\lambda(u) = \psi(\varpi^{-\lambda}u\varpi^{\lambda})$, and \begin{equation}\label{eq:miw}
        \phi_w(\lambda,w'; z) = \int_{U} f_w(uw')\psi_\lambda(u)\,du.
    \end{equation}
    
    This integral $\phi_w(\lambda,w'; z)$ is the main object of our study as it gives all the values of the Iwahori Whittaker functions. The recursive calculation of it is given by Brubaker, Buciumas, Bump, and Gustafsson \cite{BBBG19, BBBG20}. The main goal of present paper is to calculate these values combinatorically. We do it in \Cref{section:calculation}. 
    
    Notice that the \cref{eq:miw} is an integral over $U$ which lives in the derived group of $G$. Hence, all the computations can be reduced to the derived group. We will need the following generators.
    
    \subsection{Generators of the derived group}\label{section:generators}
    
    Let $G' = [G, G]$ be the derived group of $G$, so $G'$ is semisimple. Let $\widetilde{G}'$ be the induced covering group. We mainly work with $\widetilde{G}'$ since all the calculations happen in the subgroup $\widetilde{U}^-$ of $\widetilde{G}'$. Denote $\widetilde{T}' = \widetilde{T} \cap \widetilde{G}'$, $\widetilde{B}' = \widetilde{B} \cap \widetilde{G}'$, and $\widetilde{K}' = \widetilde{K} \cap \widetilde{G}'$. The Iwasawa decomposition $G = BK$ lifts to $\widetilde{G} = \widetilde{B}\widetilde{K}$, and hence $\widetilde{G}' = \widetilde{B} '\widetilde{K}'$. By Section 4, \cite{McN12}, $\widetilde{G}$ splits over $U$, $U$, and $K$. Thus, we can identify $\widetilde{U}, \widetilde{U}^-, \widetilde{K}$ with their images in $G$ and in $G'$. Also, for $\alpha \in \Phi$ and $x \in \OO_F$, we have $e_{\alpha}(x) \in \widetilde{K}$, as one would expect.
    
    The group $\widetilde{G}'$ is a quotient of the universal central extension of $G'$, and thus admits the description in terms of generators of relations. Namely, the group $\widetilde{G}'$ is generated by symbols $e_{\alpha}(x)$ where $\alpha \in \Phi$ and $x \in F$, subject to the relations 
    \[
    e_{\alpha}(x)e_{\alpha}(y) = e_{\alpha}(x+y),
    \]
    \[
    w_\alpha(x)e_{\alpha}(y)w_{\alpha}(-x) = e_{-\alpha}(-x^{-2}y),
    \]
    where $w_{\alpha}(x) =
    e_{\alpha}(x)e_{-\alpha}(-x^{-1})$, and 
    \[
    e_{\alpha}(x)e_{\beta}(y) = \left[\prod_{\substack{i,j \in \Z^+ \\ i \alpha + j \beta = \gamma \in \Phi}}e_{\gamma}(c_{i,j,\alpha,\beta}x^iy^j)\right]e_\beta(y)e_\alpha(x),
    \]
    for all $x, y \in F$ and $\alpha, \beta \in \Phi$ with $\alpha + \beta \neq 0$, where $c_{i,j,\alpha,\beta}$ is a fixed collection of integers, completely determined by the root system $\Phi$. See Section 3 of \cite{McN11} and \cite{Ste68} for details.
    
    Note that the terms in the last product commute, so there is no ambiguity with respect to order of multiplication. In the noncommutative case we write $\prod_{k=m}^{n}x_k$ for $x_mx_{m+1}\dots x_n$, and $\prod_{k=n}^{m}x_k$ for $x_nx_{n-1}\dots x_m$ where $m \leq n$. 
    
    Moreover, there are additional relations coming from the choice of the central extension. Define the elements $h_\alpha(x) \in \widetilde{G}'$ by $h_\alpha(x) = w_\alpha(x)w_{\alpha}(-1)$ and let $l_\alpha = Q(\alpha^\vee) = B(\alpha^\vee, \alpha^\vee)/2$. Then the following identities hold in $\widetilde{G}'$.
    
    \begin{equation}\label{eq:hrelations}
        \begin{gathered}
            h_{\alpha}(x)e_{\beta}(y)h_{\alpha}(x)^{-1} = e_\beta(x^{ \langle \beta, \alpha^\vee \rangle }y), \\
            h_\alpha(x)h_{\alpha}(y) = (x, y)^{l_\alpha}h_\alpha(xy), \\
            h_\alpha(x)h_\beta(y)h_\alpha(x)^{-1} = h_\beta(y)(x,y)^{ \langle \beta, \alpha^\vee \rangle l_\beta}, \\
            h_\alpha(x) = h_{-\alpha}(x^{-1}),
        \end{gathered}
    \end{equation}
    
    where we recall that $(x,y) \in \mu_n$ is the value of Hilbert symbol and is central in $\widetilde{G}$.
    
    With these generators, $\widetilde{T}'$ is generated by the images of all elements of the form $h_\alpha(x)$ for $x \in F$, $\widetilde{U}'$ is generated by the images of $e_\alpha(x)$ for $\alpha \in \Phi^+$ and $x \in F$. Similarly, $\widetilde{U}^-$ is the subgroup of $\widetilde{G}'$ generated by all $e_{-\alpha}(x)$ where $\alpha \in \Phi^+$ and $x \in F$. 
    
    Let $W = N_G(T)/T$ be the Weyl group. The induced cover $\widetilde{W}$ of $W$ splits over the maximal compact subgroup $K$. If we choose coset representatives in $K$, we can identify $W$ with $\widetilde{W}$. 
    
    Let $s_i \in K$ be the representatives of $w_{\alpha_i}(-1) \in \widetilde{W}$. Then the elements $s_i$ generate $\widetilde{W}$. For any $w \in W$ write $w = s_{i_1}\dots s_{i_m}$, then we have the element $w = s_{i_1}\dots s_{i_m}$ that maps to $w \in W$ under the projection map $p$. For any positive root $\alpha \in \Phi^+$, choose a simple root $\alpha_i$ and a Weyl group element $w \in W$ such that $\alpha = w \cdot \alpha_i$. Then set $s_{\alpha} = \widetilde{w}s_i\widetilde{w}^{-1}$.
    
    \bibliographystyle{alphaurl}
    \bibliography{bibliography}
	
\end{document}